\documentclass[a4paper,10pt]{amsart}
\usepackage[utf8]{inputenc}
\usepackage{verbatim}
\usepackage{amssymb,amsmath}
\usepackage{enumerate}
\usepackage[active]{srcltx}
\usepackage{mathabx}

\numberwithin{equation}{section}
\usepackage{float}
\usepackage[left=3.5 cm, right=3.5cm]{geometry}
 \usepackage[usenames,dvipsnames]{pstricks}
 \usepackage{epsfig}
\usepackage{pst-grad}
\usepackage{pstricks,pst-node}

\usepackage{xmpmulti}
\usepackage{psfrag}
\usepackage{tikz}

\title{Orientations of graphs with uncountable chromatic number}
\author{D. T. Soukup}
\date{\today}
 \email[Corresponding author]{daniel.soukup@univie.ac.at}
\urladdr{http://www.logic.univie.ac.at/$\sim  $soukupd73/}
\address{Universität Wien
Kurt G\"odel Research Center for Mathematical Logic
W\"ahringer Strasse 25
1090 WIEN
AUSTRIA}

\dedicatory{Dedicated to Professor Andr\'as Hajnal.}


\newtheorem{prop}{Proposition}[section]

\newtheorem{lemma}[prop]{Lemma}
\newtheorem{fact}[prop]{Fact}
\newtheorem{cor}[prop]{Corollary}
\newtheorem{con}[prop]{Conjecture}
\newtheorem{theorem}[prop]{Theorem}
\newtheorem{theorem?}[prop]{Theorem???}
\newtheorem{tclaim}{Claim}[prop]

\newtheorem{claim}[prop]{Claim}
\newtheorem{obs}[prop]{Observation}

\newtheorem{prob}[prop]{Problem}
\newtheorem{quest}[prop]{Question}

\DeclareMathOperator{\cf}{cf}
\DeclareMathOperator{\Sh}{Sh}

\newcommand{\half}{H_{\oo,\oo}}

\newcommand{\mc}[1]{\mathcal{#1}}
\newcommand{\mb}[1]{\mathbb{#1}}

\newcommand{\oo}{\omega}
\newcommand{\uhr}{\upharpoonright}
\newcommand{\omg}{{\omega_1}}
\newcommand{\dom}{\text{dom}}
\newcommand{\ran}{\text{ran}}
\newcommand{\supp}{\text{supp}}

\newcommand{\chr}[1]{\chi(#1)}
\newcommand{\dchr}[1]{\arr{\chi}(#1)}
\newcommand{\arr}[1]{\overrightarrow{#1}}
\newcommand{\larr}[1]{\overleftarrow{#1}}

\newcommand{\ENL}[3]{{#1}\stackrel{\text{ENL}}{\longrightarrow} \bigl ({#2}\bigr)^1_{#3}}
\newcommand{\ENLL}[2]{{#1}\stackrel{\text{ENL}}{\Longrightarrow} \bigl({#2}\bigr)}
\newcommand{\force}{{\hspace{0.02 cm}\Vdash}}
\newcommand{\setm}{\setminus}

\newcommand{\DD}{\color{black}}

\begin{document}

\begin{abstract}Motivated by an old conjecture of P. Erd\H os and V. Neumann-Lara, our aim is to investigate digraphs with uncountable dichromatic number and orientations of undirected graphs with uncountable chromatic number.  {\DD A graph has uncountable \emph{chromatic number} if its vertices cannot be covered by countably many independent sets, and a digraph has uncountable \emph{dichromatic number} if its vertices cannot be covered by countably many acyclic sets.}  We prove that consistently there are digraphs with uncountable dichromatic number and arbitrarily large {\DD digirth}; this is in surprising contrast with the undirected case: any graph with uncountable chromatic number contains a 4-cycle. Next, we prove that several well known graphs (uncountable complete graphs, certain comparability graphs, and shift graphs) admit orientations with uncountable dichromatic number in ZFC. However, we show that the statement ``every graph $G$ of size and chromatic number $\omg$ has an orientation $D$ with uncountable dichromatic number" is independent of ZFC.
\end{abstract}

\maketitle

\section{Introduction}

The chromatic number of an undirected graph $G$, denoted by $\chr{G}$, is the minimal number of independent sets needed to cover the vertex set of $G$. A beautiful branch of graph theory deals with the problem of understanding the consequences of having large (finite or infinite) chromatic number. In particular, what subgraphs $H$ must appear in graphs $G$ with large, say uncountable chromatic number? Is it true that cycles, paths or certain highly connected sets must embed into every graph with large enough chromatic number? There are numerous deep results regarding these questions; the investigations started in the 1960s with a seminal paper of P. Erd\H os and A. Hajnal \cite{EH0} and later on, significant contributions were made by P. Komj\' ath, S. Shelah, C. Thomassen, S. Todorcevic and several other people. In particular, it is now well understood exactly what cycles and finite graphs must embed into a graph $G$ with $\chr{G}>\oo$.  We shall review some of these results in later sections but the surveys \cite{kopesurv, kopeerdos} offer great overview of this topic.

 In the case of directed {\DD graphs}, acyclic sets play the role of independent sets: the \emph{dichromatic number of a directed graph} $D$, denoted again by $\chr{D}$, is defined to be the minimal number of \emph{acyclic vertex sets} needed to cover the vertices of $D$ \cite{nlara}. {\DD The notion of the dichromatic number of digraphs is certainly well investigated (see \cite{dutta, harut,harutplanar,mohar2016,spencer2008size} for various directions in research).} Now, our paper is motivated by two fundamental questions: first, we aim to understand which classical results on chromatic number and obligatory subgraphs extend to the directed case. Second, we hope to shed more light on an old conjecture of Erd\H os and V. Neumann-Lara \cite{ENL, nessparse}:

\begin{con}\label{ENLconj} There is a function $f:\mb N\to \mb N$ so that $\chr G\geq f(k)$ implies that $\chr{D}\geq k$ for some orientation $D$ of $G$.
\end{con}

Note that any graph $G$ with $\chr G\geq 3$ must contain a cycle and hence there is an orientation $D$ of $G$ with a directed cycle i.e. $\chr D\geq 2$. In turn $f(2)=3$ but no other value of the function $f$ is currently known. Our  aim will be to understand the possible values of $\chr{D}$ where $D$ is an orientation of a graph $G$ with $\chr{G}>\oo$. In \cite{ENL}, a related invariant is introduced and further investigated in {\DD \cite{erdos_dichrom}}: let $$\dchr{G}=\sup\{\chr{D}: D \text{ is an orientation of }G\}.$$  That is, $\dchr{G}\geq k$ means that there is an orientation of $G$ so that whenever we colour the vertices of $G$ with $<k$ colours then we can find a monochromatic directed cycle.  As Erd\H os noted in \cite{ENL}, it is surprisingly hard to determine $\dchr{G}$ for rather simple graphs $G$; we certainly can't refute this in the case of uncountable graphs either. 

\medskip

Before we summarize the results of our paper, let us introduce some notation: throughout the paper, $G$ will denote an undirected graph and $D$ a digraph. An orientation $D$ of an undirected graph $G$ is a digraph $D$ with the same set of vertices as $G$ and for every undirected edge $ab$ in $G$ either $ab$ or $ba$ (but not both) is an arc of $D$. We will use the well known arrow notation:
$$D\to (D_0)^1_r$$ means that for every $r$-colouring of the vertices of $D$ one can find a monochromatic copy of $D_0$. The negation is denoted by $D\nrightarrow (D_0)^1_r$. 

 We let $$\ENL{G}{D_0}{r}$$ mean that there is an orientation $D$ of $G$ such that $D\to (D_0)^1_r$.

We will write $$\ENLL{G}{D_0}$$ to denote the fact that there is an orientation $D$ of $G$ such that {\DD $D_0$ is a subgraph of $D[W]$} whenever $\chr{G[W]}=\chr{G}$.

\medskip

We start in Section \ref{prelimsec} by proving an important lemma on amalgamating digraphs with large digirth; this will later be applied in multiple arguments. Next, in Section \ref{obligsec}, we investigate what are those directed graphs that embed into any digraph $D$ with  $\chr D>\omega$. The two main results of this section are Theorem \ref{digirth} and \ref{noarrowcon}: we prove that consistently 
\begin{itemize}
 \item for each $k<\oo$ there is a digraph $D$ with $\chr D>\oo$ so that $D$ has no directed cycles of length $\leq k$;
\item there is a digraph $D$ with $\chr D>\oo$ so that $D\nrightarrow (\arr{C}_k)^1_k$ for all $k<\oo$.
\end{itemize}
This is in surprising contrast with the undirected case: $\chr G>\oo$ implies that {\DD $G\to ({C}_{2k})^1_\oo$} for all $k<\oo$. We remark that a standard compactness argument combined with Theorem \ref{digirth} shows the existence of finite digraphs $D$ with arbitrary large digirth and dichromatic number, a result of D. Bokal et al \cite{bokal}.

Next, in Section \ref{posrelsec}, we construct various orientations of graphs $G$ with uncountable chromatic number. First, we look at specific graphs: the complete graph on $\kappa$ vertices, comparability graphs of Suslin trees and certain non-special trees and shift graphs. We show that these undirected graphs all admit orientations with large dichromatic number (in ZFC). Now, we can see that any obligatory subgraph for digraphs $D$ with uncountable dichromatic number must be bipartite and consistently acyclic. 

 Second, we show in Theorem \ref{posrel} that  $\chr G=\omg$ is equivalent to $\dchr G=\omg$ under $\diamondsuit^+$ for any $G$ of size  $\omg$. Actually, we prove the much stronger relation $$\ENLL{G}{D}$$ where $D$ is any orientation of the half graph $H_{\omega,\omega}$.  

Finally, in Section \ref{negrelsec}, we show that consistently there is a graph $G$ with $\chr G=|G|=\omg$ but $\dchr G\leq \oo$; that is, $\chr D\leq \oo$ for any orientation $D$ of $G$. In particular, this provides some information on the  Erd\H os-Neumann-Lara conjecture for uncountable graphs: the statement ``$\chr G=\omg$ implies $\dchr G=\omg$ for $G$ of size $\omg$" is independent of ZFC. 

We end our paper with a healthy list of open problems which in our opinion worth the attention of the interested reader.
\medskip

\textbf{Acknowledgements.} Several arguments in the present paper were motivated by ideas from \cite{simchrom}.  {\DD We thank the anonymous referees for their careful reading and many advice which significantly improved the presentation of the paper. 

The author was supported in part by PIMS and the FWF Grant I1921.}

\section{Preliminaries}\label{prelimsec}

{\DD In our paper, $G$ always denotes an undirected graph i.e. a pair $(V,E)$ so that $E\subseteq [V]^2$. A pair $D=(V,E)$ is a digraph if $E\subseteq V^2$, and we do not allow multiple arcs i.e. if $uv\in E$ then $vu\notin E$. We say that \emph{$D$ is an orientation of $G$} if $D$ and $G$ have the same set of vertices and $D$ has an arc between two vertices $u$ and $v$ (in exactly one of the two directions) if and only if $uv$ is an edge in $G$.}

 We will use $V(G)$ and $V(D)$ to denote the vertex set of $G$ and $D$, and $E(G)$ and $E(D)$ to denote the edge/arc set of $G$ and $D$, respectively. We let $N^+(v)=\{w\in V(D):vw\in E(D)\}$ and $N^-(v)=\{w\in V(D):wv\in E(D)\}$. For digraphs $D_i$, we use the convention that $D=\bigcup \{D_i:i<n\}$ is the pair $(\bigcup \{V(D_i):i<n\},\bigcup \{E(D_i):i<n\})$ which may or may not be a digraph {\DD in our definition (since multi-edges could be introduced).}

{\DD We write $G_0\hookrightarrow G$ to denote the fact that $G_0$ embeds into $G$ as a not necessarily induced subgraph; $\hookrightarrow$ will also be used in the context of digraphs.} We let $G[W]$ and $D[W]$ denote the induced subgraph of $G$ and $D$ on vertices $W$. 


We say that the length of a path is the number of its edges. Let $\arr{P_\omega}$ denote the one way infinite directed path and let $\arr{C_n}$ denote the directed cycle with $n$ vertices. {\DD The \textit{girth/digirth of a graph/digraph} is the length of its shortest cycle/directed cycle.}
\medskip

We will frequently use the following lemma on amalgamating digraphs with prescribed digirth.

\begin{lemma}\label{amalglemma}
 Suppose that the digraphs $D_i$ are on vertex sets $V_i$ (finite or infinite) so that {\DD there is a single $R$ such that} $R=V_i\cap V_j$ and there is a digraph isomorphism $\psi_{i,j}:V_i\to V_j$ which is the identity on $R$ for all $i<j<n$. Then $D=\bigcup \{D_i:i<n\}$ is a digraph. 

Fix $k\in \oo$ at least 3. If each $D_i$ has digirth bigger than $k$ then
\begin{enumerate}
 \item any path  $P$ from $\alpha\in V_i$ to $\alpha'=\psi_{i,j}(\alpha)\in V_j$ in $D$ has length $>k$;
\item $D$ has digirth bigger than $k$.
\end{enumerate}
Furthermore, suppose that $\alpha_{i}\in V_i\setm R$ so that $\alpha_j=\psi_{i,j}(\alpha_i)$ for $i<j<n$. 
\begin{enumerate}
\setcounter{enumi}{2}
 \item Let $n>k$ and define $D^*$ by $V(D^*)=V(D)$ and $E(D^*)=E(D)\cup \{\alpha_{n-1}\alpha_0,\alpha_i\alpha_{i+i}:i<n-1\}$. Then $D^*$ has digirth bigger than $k$.

\end{enumerate}
\end{lemma}

Note that the analogue of Lemma \ref{amalglemma} trivially fails for undirected graphs: it is easy to find $G_0,G_1$ both copies of the path of length 2 so that $G_0\cup G_1$ is a copy of $C_4$.

\begin{proof} It is obvious that $D$ is a digraph.

 (1) Suppose that there is a path $P$ on vertices $a_0=\alpha,a_1, \dots ,a_{\ell-1},a_{\ell}=\alpha'$ from $\alpha\in V_{i^*}$ to $\alpha'=\psi_{i^*,j^*}(\alpha)\in V_{j^*}$ in $D$ which has length $\ell\leq k$; we can suppose that $\ell$ is minimal. Let $\psi_{i,i}$ be the identity on $V_i$ and let $$\psi=\bigcup\{\psi_{i,j^*}:i< n\}.$$  Note that $\psi$ is a digraph homomorphism from $D$ to $D_{j^*}$. Furthermore, $\psi$ is injective on $\{a_i:i<\ell\}$ by the minimality of $\ell$. Hence $\psi(a_0)=\alpha', \dots ,\psi(a_{\ell-1}),\psi(a_{\ell})=\alpha'$ is a cycle in $D_{j^*}$ of length $\ell\leq k$ which contradicts that $D_{j^*}$ has digirth $>k$.

(2) Now, suppose that $C$ on vertices $a_0,a_1, \dots ,a_{\ell-1},a_{\ell}=a_0$ is a cycle in $D$ of length $\ell\leq k$. Let $j^*\in n$ so that $a_0\in V_{j^*}$. If $\psi$ is defined as above then $\psi$ has to be injective on $\{a_i:i<\ell\}$ otherwise there is a path (a subgraph of $C$) contradicting (1). In particular, $\psi(a_0)=a_0, \dots ,\psi(a_{\ell-1}),\psi(a_{\ell})=a_0$ is a cycle in $D_{j^*}$ so $\ell>k$; this is a contradiction.

(3) Suppose that $C$ on vertices $a_0, \dots ,a_{\ell-1},a_{\ell}=a_0$ is a cycle in $D^*$ of length $\ell\leq k$. (1) and (2) {\DD imply} that $C$ must contain at least 2 non adjacent edges from $D^*\setm D$. Also, as $k<n$, there must be a vertex of $C$ not in $A=\{\alpha_i:i<n\}$. Hence, for some $\ell_0<\ell_1<\ell$, $a_{\ell_0},a_{\ell_1}\in A$ and $a_{\ell_0},\dots ,a_{\ell_1}$ is a directed path in $D$. However, this (and $\ell\leq k$) contradicts (1).
\end{proof}

Finally, let us slightly extend the arrow notations: given a set of directed graphs $\mc D$ we let $$D\to (\bigwedge \mc D)^1_r$$ mean that  for every $r$-colouring of the vertices of $D$ and every $D_0\in \mc D$ there is a monochromatic copy of $D_0$ in $D$. Similarly, $$D\to (\bigvee \mc D)^1_r$$ means that  for every $r$-colouring of the vertices of $D$ there is a monochromatic copy of some digraph $D_0$ from $\mc D$ in $D$.

Now, we write $$\ENL{G}{\bigvee \mc D}{r}$$ to mean that 
there is an orientation $D$ of $G$ such that $D\to (\bigvee \mc D)^1_r$. So the relation $\dchr G >\omega$ can be written as $$\ENL{G}{\arr{C}_3\vee\arr{C}_4\vee \dots}{\omega}$$  or $\ENL{G}{\bigvee_{3\leq n<\oo} \arr{ C}_n}{\omega}.$ 

Let us omit the straightforward definitions of $\ENL{G}{\bigwedge \mc D_0}{r}$, $\ENLL{G}{\bigvee \mc D_0}$ and $\ENLL{G}{\bigwedge \mc D_0}$. 

{\DD
\subsection{Set theoretic preliminaries} In general, we use standard set theoretic notations and definitions but let us refer the reader to \cite{kunen} for anything that is left undefined. However, we do include a short reminder of two key (and somewhat advanced) concepts that appear regularly: elementary submodels and forcing.

 First, we say that a subset $M$ of a model $V$ is an elementary submodel if for any first order formula $\phi$ with parameters from $M$ is true in $(M,\in)$ (written as $M\models \phi$) if and only if it is true in $(V,\in)$. We write $M\prec V$ in this case. For technical reasons, one takes elementary submodels of $H(\theta)$, the collection of sets of hereditary cardinality $<\theta$, for a large $\theta$, instead of the complete set theoretic universe $V$. 
 
 The idea is, that if $G$ is an uncountable graph and $M$ is a countable elementary submodel so that $G=(V,E)\in M$, then the countable graph $G\uhr M:=G[V\cap M]$ highly resembles the uncountable $G$. Now $(M_\xi)_{\xi<\zeta}$ is a continuous chain of models if $M_\nu\subseteq M_\xi$ for $\nu<\xi$ and $M_\xi=\bigcup_{\nu<\xi}M_\nu$ for any limit $\xi<\zeta$. If $(M_\xi)_{\xi<\zeta}$ covers a graph $G$ then we can get a very useful decomposition of $G$ by looking at $(G\uhr M_{\xi+1}\setm M_\xi)_{\xi<\zeta}$. 
 
 Probably the most useful thing to keep in mind is the following:
 
  \begin{fact}\label{elfact} Suppose that $M$ is a countable elementary submodel of $H(\theta)$ and $A\in M$. If $A$ is countable then $A\subseteq M$ or equivalently, if $A\setm M$ is nonempty then $A$ is uncountable.
\end{fact}

In particular, if $G\in M$ and a finite set of vertices $W\subseteq G\uhr M$ has a single common neighbour outside $G\uhr M$ then there must be uncountably many common neighbours to $W$ in $G$ (and infinitely many of these will be in $G\uhr M$ too). Let us refer the reader to \cite{soukel} for a complete introduction to elementary submodels and combinatorics.

\medskip

Next, our main tool to prove the consistency of a statement is either invoking a combinatorial principle (like $\diamondsuit^+$) or by forcing. With forcing, one looks at a (countable) model $V$ of ZFC and a poset $\mb P\in V$ to form a larger model $V^{\mb P}$ by adding a  filter $\mc G\subseteq \mb P$ which is generic with respect to $V$. For example, $\mb P$ can be the set of all finite graphs (with a certain property) on say $\omg$; when extending a graph $p\in \mb P$ to a larger graph $q\in \mb P$, we do not add new edges between vertices of $p$. Now any filter $\mc G\subseteq \mb P$ defines a graph $G=\bigcup \mc G$ which, in the case of a generic filter, is a quite random and useful object.

A key property of forcing is that any formula $\phi$ which is true in the extension (i.e. $V^{\mb P}\models \phi$) is \emph{forced} by a condition  $p$ from the filter $\mc G$ (written as $p\force \phi$). Finally, in order to show that the forcing behaves nicely (i.e. no cardinals are collapsed) we will prove that our posets are ccc i.e. any set $Q\subset \mb P$ of uncountably many conditions contains $p\neq q\in Q$ with a common extension. The way to do this (in our case) is to find $p\neq q\in Q$ which are isomorphic and agree on their common vertices; this is done generally by applying the $\Delta$-system lemma and Lemma \ref{amalglemma}.

\begin{fact}[$\Delta$-system lemma]
 Suppose that $\mc S$ is an uncountable set of finite sets. Then there is a single finite set $r$ and uncountable $\mc R\subset \mc S$ so that $s\cap t=r$ for any $s\neq t\in \mc R$.
\end{fact}

Naturally, one can suppose that all elements of $\mc R$ have the same size and, in case of finite graphs, each $s\in \mc R$ carries the same graph.

}

\section{Obligatory subgraphs of digraphs with uncountable dichromatic number}\label{obligsec}

For directed graphs $D$, we can ask what implications does $\chr D>\omega$ have; in particular, what are those directed graphs that embed into any digraph $D$ with  $\chr D>\omega$? We will mention the undirected counterparts of our results as we proceed.

\begin{prop}
 Suppose that $\chr D >\omega$. Then there is $D_0\subseteq D$ so that $\chr{D_0}>\omega$ and each vertex in $D_0$ has infinite in and out degree.
\end{prop}

{\DD We thank one of our anonymous referees for simplifying the original proof of this result.}

\begin{proof}
 Suppose that $D$ is a counterexample to the statement with minimal cardinality; in particular, any subgraph of $D$ with size $<|D|$ has countable dichromatic number. Now, we can find an ordinal $\kappa$ and vertices $\{v_\alpha:\alpha<\kappa\}$ so that $\chr{D[V\setm \{v_\alpha:\alpha<\kappa\}]}\leq \omega$ and {\DD for all  $\alpha\in \kappa$} either 
\begin{equation}\label{c1}
 |\{\beta\in \kappa\setm \alpha: {v_\alpha v_\beta}\in E(D)\}|<\oo
\end{equation}
or 
\begin{equation}\label{c2}
|\{\beta\in \kappa\setm \alpha: {v_\beta v_\alpha}\in E(D)\}|<\oo.
\end{equation} 
Indeed, given vertices $v_\alpha$ for $\alpha<\beta$ we look at $D[V\setm \{v_\alpha:\alpha<\beta\}]$: if this digraph has countable dichromatic number then we stop. Otherwise, there must be a vertex $v_\beta\in V\setm\{v_\alpha:\alpha<\beta\}$ so that $v_\beta$ has only finitely many in or finitely many out neighbours in $D[V\setm \{v_\alpha:\alpha<\beta\}]$.

 We let $V^+$ and $V^-$ denote the set of $v_\alpha$ so that (\ref{c1}) or (\ref{c2}) above holds, respectively. 

Now, it suffices to show that $\chr{D[V^+]}\leq \omega$ and $\chr{D[V^-]}\leq \omega$ holds. This, together with $\chr{D[V\setm \{v_\alpha:\alpha<\kappa\}]}\leq \omega$ implies that $\chr D\leq \oo$ which is a contradiction.

{\DD Consider $V^+$ and the set-mapping $F^+$ defined by $v_\alpha\mapsto \{v_\beta\in N^+(\alpha):\alpha<\beta\}$. By Fodor's theorem \cite{fodor}, $V^+$ is the union of countably many $F^+$-free sets $\{V^+_i:i<\oo\}$ i.e. $u\notin F^+(v)$ if $u\neq v\in V^+_i$. In other words, each arc of $D[V^+_i]$ goes down with respect to the well order we defined and so $D[V^+_i]$ is acyclic. The argument for $V^-$ is completely analogous.}

\end{proof}

\begin{cor}\label{treeembed}
 $\arr{P_\omega}$ embeds into $D$ whenever $\chr D >\omega$. Moreover, if $T$ is any orientation of the everywhere $\oo$-branching rooted tree then $T$ embeds into $D$ whenever $\chr D >\omega$.
\end{cor}

The undirected version of the above lemma and corollary appeared in \cite{EH0} and we followed similar proofs.

Before proceeding further, we mention that the set of obligatory digraphs for graphs with $\chr D >\omega$ is closed under a simple operation: let $\text{rev}(D_0)$ denote the digraph on vertices $V(D_0)$ and edges $\{uv:vu\in E(D_0)\}$.

\begin{obs}\label{revobs}
 If $D_0\hookrightarrow D$ for every $D$ such that $\chr D >\omega$ then {\DD $\text{rev}(D_0)\hookrightarrow D$ for every $D$ such that $\chr D >\omega$ as well.}
\end{obs}
\begin{proof}
 
Indeed, note that $\chr{\text{rev}(D)}=\chr{D}$ so $D_0\hookrightarrow \text{rev}(D)$ as well which implies that {\DD $\text{rev}(D_0)\hookrightarrow \text{rev}(\text{rev}(D))=D$.}

\end{proof}

One of the strongest results on obligatory subgraph was found by A. Hajnal and P. Komj\'ath: the half graph $H_{\omega,\omega}$ embeds into any graph $G$ with $\chr G>\omega$ \cite{half}. Recall that $H_{\omega,\omega}$ is the graph defined on vertices $\oo\times 2$ and $(k,i)(\ell,j)$ is an edge if and only if $k\leq \ell<\oo$ and $i=0,j=1$.

There are two simple orientations of $\half$: $(k,0)(\ell,1)$ is an arc if and only if $k\leq \ell<\oo$ or $(\ell,1)(k,0)$ is an arc if and only if $k\leq \ell<\oo$. We will denote these graphs by $\arr{\half}$ and $\larr \half$, respectively.

\begin{prop}\label{halfembed}
 $\arr{\half}$ and $\larr \half$ both embed into $D$ if $\chr{D}>\oo$.
\end{prop}
\begin{proof}
It suffices to prove for  $\arr{\half}$ by Observation \ref{revobs}. Suppose that $D$ is a digraph on vertex set $V$ without a copy of $\arr{\half}$ so that  $\chr{D}>\oo$. Let us also suppose that $D$ has minimal size among these graphs. Cover $D$ by a continuous chain of elementary submodels $(M_\xi)_{\xi<\zeta}$ so that $|M_\xi|<|D|$ and $D\in M_\xi$.

\begin{tclaim}
 If $v\in V\cap M_{\xi+1}\setm M_\xi$ then $N^-(v)\cap M_\xi$ is finite.
\end{tclaim}

Indeed, suppose that $x_0,x_1\dots \in N^-(v)\cap M_\xi$. The set $N^+[\{x_i:i<n\}]$ must be uncountable otherwise  $N^+[\{x_i:i<n\}]\subseteq M_\xi$ and so $v\in M_\xi$. Hence, we can find distinct $y_0,y_1\dots $ so that $y_n\in N^+[\{x_i:i<n\}]$. Now $\arr{\half}\hookrightarrow D[\{x_i,y_i:i<\oo\}]$. This contradicts our assumption that $\arr{\half}$  does not embed into $D$.


By the minimal size of $D$, there are maps $f_\xi:V\cap M_{\xi+1}\setm M_\xi\to \oo$ so that there are no monochromatic cycles with respect to $f_\xi$. Define $f=(f^1,f^2):V\to \oo\times \oo$ so that $f^1=\bigcup_{\xi<\zeta} f_\xi$ and $f^2(v)\neq f^2(w)$ if $v\in V\cap M_{\xi+1}\setm M_\xi$ and $w\in N^-(v)\cap M_\xi$. This can be done as $N^-(v)\cap M_\xi$ is finite.

We claim that $f$ witnesses that $\chr D\leq \oo$ which is a contradiction. Indeed, the definition of $f^1$ guarantees that if $C$ is monochromatic with respect to $f$ then $C$ must have an arc of the form $wv$ with $v\in V\cap M_{\xi+1}\setm M_\xi$ and $w\in N^-(v)\cap M_\xi$. But in this case $f^2(v)\neq f^2(w)$
\end{proof}

At this point, we are uncertain of exactly what orientations of $\half$ must embed into any $D$ with $\chr{D}>\oo$. 


\subsection{Cycles and dichromatic number}

Erd\H os proved in the groundbreaking \cite{prob} that there are graphs with arbitrary large finite chromatic number and arbitrary large girth. Rather surprisingly this fails for uncountable chromatic number: if $\chr G>\oo$ then $G$ contains a 4-cycle. This was originally proved in \cite{EH0} but also follows from the fact that $\half$ embeds into $G$ if $\chr G>\oo$.

Now, for finite directed graphs the analogue of Erd\H os' thereom was proved by Bokal et al \cite{bokal}: there are digraphs with arbitrary large finite dichromatic number without short directed cycles. At this point, it is somewhat unexpected that this result extends to uncountably dichromatic directed graphs as well:

\begin{theorem}\label{digirth}
Consistently, for each natural number $n\geq 3$ there is a digraph $D$ on vertex set $\omg$ so that
\begin{enumerate}
 \item $D$ has {\DD digirth} bigger than $n$, and
\item  $\arr{C}_{n+1}\hookrightarrow D[X]$ for every uncountable $X\subseteq \omg$.
\end{enumerate}
In particular, $\chr D=\omg$. 
\end{theorem}

\begin{proof}
 We show that for any $n$ there is a ccc poset of size $\omega_1$ which introduces such a digraph $D$. We leave it to the reader to check that the finite support product or iteration of these countably many posets gives a model with the appropriate graphs for each $n$ at the same time.

Fix $n\geq 3$ and simply let $\mb P$ be the set of all finite digraphs on a subset of $\omg$ which avoid $\arr{C_k}$ for $3\leq k\leq n$ i.e. each $p\in\mb P$ is a finite digraph $(V(p),E(p))$ with digirth $>n$. We write $p\leq q$ for $p,q\in \mb P$ if $V(p)\supseteq V(q)$ and $p[V(q)]=q$.

We say that $p,q\in \mb P$ are twins if $|V(p)|=|V(q)|$, $V(p)\cap V(q)<V(p)\setm V(q)<V(q)\setm V(p)$ (or vica versa $V(q)\setm V(p)<V(p)\setm V(q)$) and the unique order preserving map $\psi_{p,q}$ from $V(p)$ to $V(q)$ is a digraph isomorphism of $p$ and $q$. Note that if $p,q$ are twins then $p\cup q$ is a digraph as well.

Clearly, any generic filter $\mc G\subseteq \mb P$ gives a digraph $\dot D$ on vertex set $\omg$ with $E(\dot D)=\bigcup\{E(p):p\in \mc G\}$.

\begin{tclaim}\label{cccclaim}
 $\mb P$ is ccc.
\end{tclaim}
\begin{proof} Note that any uncountable set of conditions contains an uncountable subset of pairwise twins by the $\Delta$-system lemma. Now, we claim that $p\cup q$ is a condition if $p,q$ are twins. Indeed, apply Lemma \ref{amalglemma} (2).

\end{proof}

The next claim finishes the proof of the theorem:

\begin{tclaim}\label{amalgclaim}
 $V^{\mb P}\models \arr{C}_{n+1}\hookrightarrow \dot D[\dot X]$ for every uncountable $\dot X\subseteq \omg$.
\end{tclaim}
\begin{proof} Suppose that $p\force |\dot X|=\omg$. There is an uncountable $Y\subseteq \omg$ and $p_\beta\in \mb P$ for $\beta \in Y$ so that $p_\beta\leq p$, $ \beta \in V(p_\beta)$ and $p_\beta\force \beta\in \dot X$. Apply the $\Delta$-system lemma to find $\alpha_0<\alpha_1<\dots<\alpha_n\in Y$ so that $p_{\alpha_i}$ and $p_{\alpha_j}$ are twins whenever $i<j<n+1$.

Define $q$ by letting
\begin{equation} 
V(q)=\bigcup_{i<n+1}V(p_{\alpha_i}) \text{ and } E(q)=\bigcup_{i<n+1}E(p_{\alpha_i})\cup \{\alpha_n\alpha_0,\alpha_i\alpha_{i+1}:i<n\}.
\end{equation}
 Lemma \ref{amalglemma} (3) implies that $q\in \mb P$ and of course $q\leq p_{\alpha_i}$. We clearly have $q\force \dot D[\alpha_0\dots\alpha_n]\hookrightarrow D[\dot X]$ and that $q\force \dot D[\alpha_0\dots\alpha_n]$ is an induced copy of $\arr{C}_{n+1}$.

\end{proof}

\end{proof}

\begin{cor} {\DD Consistently, any digraph $D_0$ which embeds into all digraphs $D$ with $\chr D>\oo$ must be acyclic.}
\end{cor}

{\DD We don't know at this point how to construct digraphs with uncountable dichromatic number but with arbitrary large digirth in ZFC.}

Now, the fact that $C_4$ appears in every graph $G$ with $\chr G>\oo$ shows that the relation $$G \to (C_4)^1_\oo$$ is equivalent to $\chr G>\oo$. The digraph version is (consistently) false by the above theorem, however at this point it seems possible that $\chr D>\oo$ implies  $D\to (\arr{C}_k)^1_\oo$ for some $k<\oo$ for any $D$. We show now that this is not the case.  Let us denote the set of nonzero, nondecreasing $f:\mb N\to \mb N$ so that $\displaystyle{\lim_{k\to \infty} f(k)=\infty}$ with $\mc F$ for the next proof.

\begin{theorem}\label{noarrowcon}
 Consistently, for $f\in \mc F$ there is a digraph $D=D_f$ on vertex set $\omg$ so that
\begin{enumerate}
 \item $\chr{D}=\omg$, and
\item  $D\nrightarrow (\arr{C}_{k})^1_{f(k)}$ for all $k<\oo$.
\end{enumerate}
\end{theorem}

\begin{proof} 
 
Given $f\in \mc F$ we define the poset $\mb P_f$ of conditions $p=(d^p,(g^p_k)_{3\leq k\leq n^p})$ where 
\begin{enumerate}[(P1)]
 \item $d^p=(V^p,E^p)$ is a finite digraph on $\omg$ and $n^p=|V^p|$, 
\item\label{p2}  $g^p_k:V(d^p)\to f(k)$, and
\item\label{p3} $d^p[\{v: g^p_k(v)=i\}]$ has digirth $>k$ for all $i<f(k)$ and $3\leq k\leq n^p$.
\end{enumerate}

We let $p\leq q$ if 
\begin{enumerate}[(i)]
 \item $V^p\supseteq V^q$ and $d^p[V^q]=d^q$,
\item  $g^q_k=g^p_k\uhr V^q$ for all $3\leq k\leq n^q$.
\end{enumerate}

It is clear that a generic filter $\mc G\subseteq \mb P$ introduces a digraph $\dot D=\bigcup\{d^p:p\in \mc G\}$  and functions $\dot g_k$ by $\dot g_k=\bigcup\{g^p_k:p\in \mc G, k\leq n^p\}$ for $k\in \oo$.

\begin{tclaim} The following holds for any generic filter $\mc G\subseteq \mb P$ and $\dot D, \dot g_k$ defined as above:
 \begin{enumerate}[(a)]
  \item $\dot D$ has vertex set $\omg$,
\item $\dom(\dot g_k)=\omg$, and
\item  $\dot D[\{v\in \omg: \dot g_k(v)=i\}]$ has digirth $>k$ for all $i<f(k)$ and $3\leq k<\oo$.
 \end{enumerate}

\end{tclaim}
\begin{proof} (a) Indeed, it suffices to show that the set $\{p\in \mb P:v\in V^p\}$ is dense for every $v\in \omg$. Given any $q\in \mb P$ and $v\in \omg\setm V^q$ we let $V^p=V^q\cup \{v\}$ and $E^p=E^q$. Then simply define $g^p_k=g^q_{k^*}\cup \{(v,0)\}$ where $3\leq k\leq n^p=n^q+1$ and $k^*=\min\{k,n^q\}$. It is easy to check that property (P\ref{p3}) is satisfied by $p$.

(b) follows from (a), and (c) follows from property (P\ref{p3}).

\end{proof} 
We say that two conditions $p,q$ are twins if 
\begin{enumerate}
 \item $n^p=n^q$ and $V(d^p)\cap V(d^q)<V(d^p)\setm V(d^q)<V(d^q)\setm V(d^p)$ (or vica versa $V(d^q)\setm V(d^p)<V(d^p)\setm V(d^q)$), 
\item the unique order preserving map $\psi_{p,q}$ from $V(d^p)$ to $V(d^q)$ is an isomorphism of the digraphs $d^p$ and $d^q$, and
\item $g^p_k(v)=g^q_k(\psi_{p,q}(v))$ for all $3\leq k\leq n^p$ and $v\in V^p$.
\end{enumerate}


\begin{claim}\label{ccc0}
 $\mb P$ is ccc.
\end{claim}
\begin{proof}
By standard $\Delta$-system arguments, it suffices to show that if $p,q\in \mb P$ are twins then they have a common extension $r\in \mb P$. We let $d^r=d^p\cup d^q$ and define $g^r_k= g^p_{k^*}\cup g^q_{k^*}$ where $3\leq k\leq n^r$ and $k^*=\min\{k,n^p\}$. Note that $f(k^*)\leq f(k)$ for $k^*=\min\{k,n^p\}$ so $g^r_k:V^r\to f(k)$ i.e. property (P\ref{p2}) is satisfied.

We need to check that $d^r[\{v\in V^r: g^r_k(v)=i\}]$ has digirth $>k$ for all $i<f(k)$ and $3\leq k\leq n^r$. Note that $$d^r[\{v\in V^r: g^r_k(v)=i\}]=d^p[\{v\in V^p: g^p_{k^*}(v)=i\}]\cup d^q[\{v\in V^q: g^q_{k^*}(v)=i\}]$$ where $k^*=\min\{k,n^p\}$. Furthermore, the graphs $d^p[\{v\in V^p: g^p_{k^*}(v)=i\}]$ and $d^q[\{v\in V^q: g^q_{k^*}(v)=i\}]$ are isomorphic and have digirth $>k$. Hence Lemma \ref{amalglemma} (2) implies that $d^r[\{v\in V^r: g^r_k(v)=i\}]$ still has  digirth $>k$. In turn, $r$ satisfies property (P\ref{p3}) and so $r\in \mb P_f$ is a common extension of $p$ and $q$.

\end{proof}

\begin{tclaim}\label{forceamalg}
$V^\mb P\models \chr {\dot D[\dot W]}=\omg$ for any uncountable $\dot W\subseteq \omg$.
\end{tclaim}
\begin{proof}
Suppose that $p\force \dot W\subseteq \omg$ is uncountable. Find $Y\in [\omg]^\omg,n\in\oo$ and $p_\alpha\leq p$ for $\alpha\in Y$ so that

\begin{enumerate}[(i)]
 \item $\{p_\alpha:\alpha\in Y\}$ are pairwise twins (with mappings $\psi_{\alpha,\alpha'}$ witnessing this) and $n^{p_\alpha}=n\geq 2$,
\item $\alpha\in V^{p_\alpha}$ and $\psi_{\alpha,\alpha'}(\alpha)=\alpha'$ for $\alpha,\alpha'\in Y$, and
\item $p_\alpha\force \alpha\in \dot W$ for all $\alpha\in Y$.
\end{enumerate}

This can be done by the $\Delta$-system lemma. Let $N\in \mb N$ be minimal so that $f(N)>f(n)$; such a value exists as $f(k)\to \infty$ as $k\to \infty$ and $N>n$ as $f$ is nondecreasing.  Now, fix distinct $\alpha_j\in Y$ for $j<N$ and define $q$ as follows:
\begin{equation} 
V^q=\bigcup_{j<N}V^{p_{\alpha_j}} \text{ and } E^q=\bigcup_{j<N}E^{p_{\alpha_j}}\cup \{{\alpha_{N-1}\alpha_0},{\alpha_j\alpha_{j+1}}:j<N-1\}.
\end{equation} 

Now we define $g^q_k$ for $3\leq k\leq n^q$ as follows: let 

 $$g^q_k= \bigcup_{i<N}g^{p_{\alpha_j}}_{k^*}$$ if $3\leq k\leq N-1$ where $k^*=\min\{k,n\}$ and let

$$g^q_k=g^{p_{\alpha_0}}_n\uhr (V^{p_{\alpha_0}}\setm\{\alpha_0\})\cup \{(\alpha_0,f(n))\}\cup\bigcup_{1\leq j<N}g^{p_{\alpha_j}}_n$$ if $N\leq k\leq n^q.$

We need to check that $d^q[\{v\in V^q: g^q_k(v)=i\}]$ has digirth $>k$ for all $i<f(k)$ and $3\leq k\leq n^q$. If $k\leq N-1$ then $$d^q[\{v\in V^q: g^q_k(v)=i\}]=\bigcup_{j<N} d^{p_{\alpha_j}}[\{v\in V^{p_{\alpha_j}}: g^{p_{\alpha_j}}_k(v)=i\}]$$ and we can apply either  Lemma \ref{amalglemma} (2) (if $g^{p_{\alpha_j}}_k(\alpha_j)\neq i$) or Lemma \ref{amalglemma} (3) (if $g^{p_{\alpha_j}}_k(\alpha_j)=i$) to see that $d^q[\{v\in V^q: g^q_k(v)=i\}]$ has digirth $>k$.

Now, for $k$ between $N$ and $n^q$, Lemma \ref{amalglemma} (3) might not apply directly as $k>N$. We distinguish 4 cases depending on value of $i<f(k)$. If $f(n)<i<f(k)$ then $d^q[\{v\in V^q: g^q_k(v)=i\}]$ is empty so we have nothing to prove. If $i=f(n)$ then $d^q[\{v\in V^q: g^q_k(v)=i\}]=\{\alpha_0\}$ so again we have nothing to prove. Now if $i<f(n)$ but $g^{p_{\alpha_j}}_n(\alpha_j)\neq i$ then again we can apply Lemma \ref{amalglemma} (2) to see that $d^q[\{v\in V^q: g^q_k(v)=i\}]$ has digirth $>k$ as $$d^q[\{v\in V^q: g^q_k(v)=i\}]=\bigcup_{j<N} d^{p_{\alpha_j}}[\{v\in V^{p_{\alpha_j}}: g^{p_{\alpha_j}}_n(v)=i\}]$$ as before.

Finally, lets look at the case when $g^{p_{\alpha_j}}_n(\alpha_j)=i$ (if this holds for one $j$ then it holds for all $j<N$ as we are working with twin conditions). Suppose that $C$ is a cycle in $d^q[\{v\in V^q: g^q_k(v)=i\}]$ of length $\leq k$. By Lemma \ref{amalglemma} (2), $C$ must contain a new edge of the form $\alpha_j\alpha_{j+1}$ where $1\leq j<N-1$. In particular, we can find $1\leq j_0<j_1\leq N-1$ so that $C$ contains a directed path from $\alpha_{j_1}$ to $\alpha_{j_0}$ using only edges from $\bigcup_{j<N} d^{p_{\alpha_j}}[\{v\in V^{p_{\alpha_j}}: g^{p_{\alpha_j}}_n(v)=i\}]$. Let $\psi=\bigcup \{\psi_{\alpha_j,\alpha_{j_1}}:j<N\}$ where $\psi_{\alpha,\alpha}$ is the identity on $V^{p_{\alpha}}$. Now $\psi$ maps $P$ into a walk from $\alpha_{j_1}$ back to $\alpha_{j_1}=\psi(\alpha_{j_0})$ in $d^{p_{\alpha_{j_1}}}[\{v\in V^{p_{\alpha_{j_1}}}: g^{p_{\alpha_{j_1}}}_n(v)=i\}]$. Also, this walk has length at most $k$ so it must contain a cycle of length at most $k$ as well. However, this contradicts that $d^{p_{\alpha_{j_1}}}[\{v\in V^{p_{\alpha_{j_1}}}: g^{p_{\alpha_{j_1}}}_n(v)=i\}]$ has digirth $>k$.

Hence, we showed that $(g^q_k)_{3\leq k\leq n^q}$ satisfies property (P\ref{p3}) and so $q\in \mb P_f$. It is now clear that $q\force \dot D[\{\alpha_j:j<N]$ is a copy of $\arr{C}_N$ in $\dot W$.

\end{proof}

At this point, we showed that for any single $f\in \mc F$ there is a ccc extension of the ground model with the required digraph $D_f$.

Now, starting from a model of CH, we can define a finite support iteration $(\mb P_\alpha,\dot {\mb Q}_\beta)_{\alpha\leq \omg,\beta<\omg}$ of length $\omg$ where $V^{\mb P_\alpha}\models \dot {\mb Q}_\alpha=\mb P_{\dot f}$ for some ${\mb P_\alpha}$-name $\dot f$ for a function in $\mc F$. 

It follows from Claim \ref{ccc0} that each ${\mb P_\alpha}$ is ccc so we can arrange the iteration in such a way that any $f\in \mc F$ in the final model shows up at some intermediate stage i.e. $\dot {\mb Q}_\alpha=\mb P_{\dot f}$ for some $\alpha$ and appropriate name $\dot f$ for $f$. So it suffices to check that in the final model $V^{\mb P_\omg}$ we still have $\chr{D_f}=\omg$ for the graphs that we introduced by the intermediate forcings $\mb Q_\alpha$. This can be done using determined conditions and the argument in Claim \ref{forceamalg}; more precisely, we add edges to the finite approximation of ${D_f}$ in coordinate $\alpha$ as in Claim \ref{forceamalg} while not adding any new edges to graphs in other coordinates. We leave the details to the interested reader.

\end{proof}

Let us also mention the following

\begin{obs} Suppose that $D$ is a digraph.
\begin{enumerate}
 \item $D\nrightarrow (\arr{C}_k)^1_\oo$ for all $k<\oo$ implies $\chr{D}\leq 2^\oo$. 
\item The edges of $D$ can always be partitioned into two acyclic sets.
\end{enumerate}

\end{obs}

Indeed, if $f_k$ witnesses $D\nrightarrow (\arr{C}_k)^1_\oo$ then  $f:V(D)\to \oo^\oo$ defined by $v\mapsto (f_k(v))_{k\in \oo}$ witnesses $\chr{D}\leq 2^\oo$. To see (2), take an arbitrary well order on the vertices and consider the forward and backward edges.

Finally, we state without proof that the famous Erd\H os-de Bruijn compactness result also holds for directed graphs:

\begin{theorem}\label{thm:compactness}
 Suppose that any finite subgraph of the digraph $D$ has finite dichromatic number at most $k$. Then $\chr D\leq k$ as well.
\end{theorem}

In particular, we can deduce the result of Bokal et al \cite{bokal} on finite digirth and dichromatic number from our forcing result in Theorem \ref{digirth}: given a model $V$ of set theory and finite number $k$, we force to find an extension $V^{\mb P}$ with a graph $D$ with digirth $>k$ and $\chr D>\oo$. By the above compactness result, there must be a finite subgraph $D^*$ of $D$ (in $V^\mb P$) which has dichromatic number $\geq k$. However, the models $V$ and $V^{\mb P}$ have the same finite digraphs and hence $D^*\in V$ as well. Much like the probabilistic proof in \cite{bokal} this forcing argument gives no information about these sparse digraphs with large dichromatic number. A simple, recursive construction of such graphs was actually given by M. Severino \cite{severino}.

\section{Orientations of undirected graphs with large chromatic number}\label{posrelsec}

There are two trivial orientations of any undirected graph $G$ given a well order $\prec$ on the vertices: define the orientation $\arr G$ of $G$ by  $uv\in E(\arr G)$ if and only if $uv\in E(G)$ and $u\prec v$. Similarly, $\larr G$ is defined by $uv\in E(\arr G)$ if and only if $uv\in E(G)$ and $v\prec u$. 

It is well known that if $G$ has countable \textit{colouring number} i.e. $\{u\in N(v):u\prec v\}$ is finite for every vertex $v\in V(G)$ (for some well order $\prec$ of the vertices) then $\chr G\leq \oo$ (see \cite{EH0}). This yields the following observations: let $\arr S$ denote the countable star with all edges pointing out, and  $\larr S$ denote the countable star with all edges pointing in. Then the orientation $\arr G$ of $G$ witnesses $\ENL{G}{ \larr{S}}{\omega}$ while $\larr G$  witnesses $\ENL{G}{ \arr{S}}{\omega}$.

As we saw in the previous section, $H_{\omega,\omega}$ embeds into any graph $G$ with $\chr G>\omega$ \cite{half}.
In particular, the girth of $G$ is at most 4 whenever $\chr G>\omega$; so it could be the case that $$\ENL{G}{\arr{C_4}}{\omega} \text{ or even } \ENLL{G}{\arr{C_4}}$$ whenever $\chr G>\omega$. Indeed, we are going to prove this, at least for some graphs.

First, let us look at complete graphs. 
Recall that $$\kappa \nrightarrow [\kappa;\kappa]^2_2$$ means that there is a function $f:[\kappa]^2\to 2$ so that for all $A,B\in [\kappa]^\kappa$ and $i<2$ there is $\alpha \in A$ and $\beta\in B$ so that $\alpha<\beta$ and  $f(\alpha,\beta)=i$.

\begin{theorem}\label{Komgprop} Suppose that $\kappa$ is an infinite cardinal.
\begin{enumerate}
                                 \item If $\kappa$ is regular and $\kappa \nrightarrow [\kappa;\kappa]^2_2$ then $\displaystyle{\ENLL{K_\kappa}{\bigwedge_{3\leq n\in \oo}\arr{C_n}}}$.
\item If $\lambda$ is uncountable then $\displaystyle{\ENL{K_{\lambda^+}}{\bigwedge_{3\leq n\in \oo}\arr{C}_n}{\lambda}}$.

                              \end{enumerate}

\end{theorem}

Let us show a corollary first:

\begin{cor}
 $\dchr{K_\kappa}=\kappa$ for any infinite cardinal $\kappa$.
\end{cor}
\begin{proof} Recall that $\kappa^+ \nrightarrow [\kappa^+;\kappa^+]^2_2$ holds whenever $\kappa$ is a regular cardinal \cite{rinotrect}. Hence Theorem \ref{Komgprop} (1) and (2) implies that  $\dchr{K_{\kappa^+}}=\kappa^+$ for any cardinal $\kappa$. 

Now, given a limit cardinal $\kappa$ let $(\kappa_i)_{i<\cf(\kappa)}$ be a cofinal sequence of regular cardinals in $\kappa$. Let $V_i\in [\kappa]^{\kappa_i^+}$ pairwise disjoint for $i<\cf(\kappa)$. Then $K_\kappa$ restricted to $V_i$ is just a copy of $K_{\kappa_i^+}$ so we can apply Theorem \ref{Komgprop} (1) to find an orientation $D_i$ witnessing $\ENLL{K_{\kappa_i^+}}{\bigwedge_{3\leq n\in \oo}\arr{C_n}}$  on $V_i$. Putting together these digraphs $D_i$ (and orienting the edges outside arbitrarily) we defined an orientation $D$ of $K_\kappa$ that witnesses $\ENL{K_{\kappa}}{\bigwedge_{3\leq n\in \oo}\arr{C_n}}{\mu}$ for any $\mu<\kappa$. In particular, $\dchr{K_\kappa}=\kappa$ .

\end{proof} 

Our proof of Theorem \ref{Komgprop} was motivated by the proof of Theorem 8 \cite{simchrom}; we will point out further connections to \cite{simchrom} later as well, in particular in Section \ref{negrelsec}.

\begin{proof} (1) Let $f:[\kappa]^2\to 2$ witness  $\kappa \nrightarrow [\kappa;\kappa]^2_2$. Now simply define $D=(\kappa,E)$ by $\alpha\beta\in E$ if $\alpha<\beta$ and $f(\alpha,\beta)=0$, otherwise $\beta\alpha\in E$.

 First, we show that any induced subgraph of size $\kappa$ contains a copy of $\arr{C}_3$. Let $W\in[\kappa]^\kappa$. Define $W^+=\{v\in W:|N^+(v)\cap W|<\kappa\}$ and $W^-=\{v\in W:|N^-(v)\cap W|<\kappa\}$. If there is a $v\in W\setm(W^+\cup W^-)$ then by the choice of $f$ we can find $\alpha\in N^+(v)\cap W$ and $\beta \in N^-(v)\cap W$ so that $v<\alpha<\beta$ and $f(\alpha,\beta)=0$. Then $v\alpha,\alpha\beta,\beta v\in E$ so $\{v,\alpha,\beta\}$ is a copy of $\arr{C}_3$.

 Now, it suffices to show that $|W^+|=\kappa$ or  $|W^-|=\kappa$ is not possible. If $|W^+|=\kappa$ then using the regularity of $\kappa$, one can find a $Y\in [W^+]^\kappa$ so that $\alpha<\beta\in Y$ implies that $\beta\notin N^+(\alpha)$. However, $f(\alpha,\beta)=0$ for some $\alpha<\beta\in Y$ by the choice of $f$ so $\beta\in N^+(\alpha)$; this is a contradiction. The proof that $|W^-|=\kappa$ is not possible is completely analogous.

Now, fix $n\in \oo$ at least 3 and $W\in[\kappa]^\kappa$; we will find a copy of $\arr{C}_n$ in $D[W]$. Find pairwise disjoint paths $P_\xi=(\alpha^\xi_0\dots \alpha^\xi_{n-2})$ in $W$ of length $n-2$ for $\xi<\kappa$. This can be done by applying Corollary \ref{treeembed}; indeed,  we already proved that $\chr{D[W\setm\delta]}>\oo$ for any $\delta<\kappa$ so $\arr{P}_\oo\hookrightarrow D[W\setm \delta]$.

 Note that if there is a single $\xi$ so that  $N^-(\alpha^\xi_0)\cap W$ and $N^+(\alpha^\xi_{n-2})\cap W$ both have size $\kappa$ then we can extend $P_\xi$ into a copy of $\arr{C}_n$ in $W$. So suppose that this is not the case; then there is $I\in [\kappa]^\kappa$ so that either 
\begin{enumerate}[(i)]
 \item $|N^-(\alpha^\xi_0)\cap W|<\kappa$ for all $\xi \in I$, or
\item $|N^+(\alpha^\xi_{n-2})\cap W|<\kappa$ for all $\xi \in I$.
\end{enumerate}
If case (i) holds then, using that $\kappa$ is regular, we can find $J\in [I]^\kappa$ so that $\xi<\zeta\in J$ implies that $\alpha^\xi_0<\alpha^\zeta_0$ and $\alpha^\zeta_0\notin N^-(\alpha^\xi_0)$. However, this clearly contradicts the choice of $f$ as there is some $\xi<\zeta\in J$ such that $f(\alpha^\xi_0,\alpha^\zeta_0)=1$.

Similarly, if case (ii) holds then we can find $J\in [I]^\kappa$ so that $\xi<\zeta\in J$ implies that $\alpha^\xi_0<\alpha^\zeta_0$ and $\alpha^\zeta_0\notin N^+(\alpha^\xi_0)$. This again contradicts the choice of $f$.

(2) Suppose that $\lambda$ is uncountable. We fix a \textit{club guessing sequence} $\{C_\alpha:\alpha\in E^{\lambda^+}_\oo\}$, that is: $C_\alpha$ is a cofinal sequence of type $\oo$ in $\alpha$ and whenever $E\subseteq \lambda^+$ is a club in $\lambda^+$ (i.e. a closed and unbounded subset) then $C_\alpha\subseteq E$ for stationary many $\alpha\in E^{\lambda^+}_\oo$. The existence of such guessing sequences was originally proved in Claim 2.3 \cite{shelahcard} (for a detailed proof see \cite{abrahamcard}). We let $I(\alpha,0)=C_\alpha(0)$ and $I(\alpha,n)=C_\alpha(n)\setm C_\alpha(n-1)$ for $1\leq n<\oo$ where $(C_\alpha(n))_{n\in \oo}$  is the increasing enumeration of $C_\alpha$. Now, define the orientation $D$ as follows: given $\alpha<\beta\in \lambda^+$ we let $\alpha\beta\in E(D)$ if and only if $n(\alpha,\beta)$ is even where $n(\alpha,\beta)=\min\{n\in\oo:\alpha\in I(\beta,n)\}$; otherwise $\beta\alpha\in E(D)$.

We will show that given a partition $\lambda^+=\bigcup\{A_i:i<\lambda\}$ there is an $i<\lambda$ so that $D[A_i]$ contains a directed $n$-cycle for all $3\leq n\in \oo$.  Take a continuous, increasing sequence of elementary submodels $(M_\xi)_{\xi<\lambda^+}$ covering $\lambda^+$ so that $\{A_i,D:i<\lambda\}\subseteq M_\xi$ and $|M_\xi|=\lambda$ for all $\xi<\lambda^+$. Let $E=\{M_\xi\cap \lambda^+:\xi<\lambda^+\}$. $E$ is a club so there is an $i<\lambda$ and some stationary $S\subseteq A_i$ so that $C_\beta\subseteq E$ for all $\beta\in S$. Observe that $I(\beta,n)\cap A_i\neq\emptyset$ for every $\beta\in S$ and $n\in \oo$.

\begin{tclaim}\label{pathclaim} For every $n\in\oo$ at least 3 and every $\delta<\lambda^+$ there is a path $P=(\alpha_0\dots \alpha_{n-2})$ in $D[A_i\setm\delta]$ so that $|N^+(\alpha_{n-2})\cap A_i|=\lambda^+$.
\end{tclaim}
\begin{proof} We prove by induction on $n\geq 3$. If $n=3$ then let $\beta\in S\setm \delta$ and pick $\alpha_0\in I(\beta,2k)$ where $k$ is large enough so that $\delta<C_\beta(2k-1)$. We need that $|N^+(\alpha_{0})\cap A_i|=\lambda^+$; if $|N^+(\alpha_{0})\cap A_i|\leq \lambda$ and $C_\beta(2k)=M_\xi\cap \lambda^+$ then $N^+(\alpha_{0})\cap A_i\subseteq M_\xi$ by elementarity as well. However, $\beta\in N^+(\alpha_{0})\cap A_i\setm M_\xi$.

Now suppose that $n>3$, and again let $\beta\in S\setm \delta$. Using the inductive hypothesis and the fact that $C_\beta(2k)=M_\xi\cap \lambda^+$ for some $\xi<\lambda^+$ find a path $P=(\alpha_0\dots \alpha_{n-2})$ in $ I(\beta,2k)$ so that $|N^+(\alpha_{n-2})\cap A_i|=\lambda^+$ where $k$ is large enough so that $\delta<C_\beta(2k-1)$. By elementarity, we can find $\alpha_{n-1}\in I(\beta,2k)\setm\{\alpha_i:i<n-1\}$ so that $\alpha_{n-1}\in N^+(\alpha_{n-2})$. As before, it is easy to show that $|N^+(\alpha_{n-1})\cap A_i|=\lambda^+$ and so $(\alpha_0\dots \alpha_{n-1})$ is the desired path.
\end{proof}

Now, fix $3\leq n\in \oo$. Let $\beta\in S$ arbitrary and find a path $P =(\alpha_0\dots \alpha_{n-2})$ in $A_i\cap I(\beta,1)$ so that $|N^+(\alpha_{n-2})\cap A_i|=\lambda^+$. This can be done by applying Claim \ref{pathclaim} with $\delta=C_\beta(0)$ inside the appropriate elementary submodel $M_\xi$ where $C_\beta(1)=M_\xi\cap \lambda^+$.

Note that $C_\beta(2)=M_{\xi'}\cap \lambda^+$ for some $\xi'<\lambda^+$ so we can find $\gamma\in N^+(\alpha_{n-2})\cap A_i\cap I(\beta,2)$. Now, $\beta\alpha_0\dots \alpha_{n-2}\gamma$ is a copy of $\arr{C}_n$ in $A_i$.

\end{proof}

{\DD Recall that given a poset $\mb P$ we define its comparability graph $G_{\mb P}$  on vertex set $\mb P$ and let $st\in E(G_{\mb P})$ if and only if $s<_{\mb P}t$ or  $t<_{\mb P}s$. A Suslin-tree  is a poset $\mb S$ so that each $p\in \mb S$ has a well ordered set of predecessors and each chain and antichain of  $\mb S$ is countable. Suslin-trees exist in some models of ZFC (e.g. if $\diamondsuit$ holds) and do not exist in others (e.g. if Martin's axiom holds without CH).}

We can use the argument from Theorem \ref{Komgprop} and a trick due to J. Steprans to get the following:

\begin{prop} Suppose that $\mb S$ is a Suslin-tree and $G_{\mb S}$ is its comparability graph. Then $\displaystyle{\ENLL{G_{\mb S}}{\bigwedge_{3\leq n\in \oo}\arr{C_n}}}$.
\end{prop}

\begin{proof} We work in a model $V$ of $ZFC$ with a Suslin tree $\mb S$ on $\omega_1$. As before, we pick $f:[\omega_1]^2\to 2$ with the property that $f''\{\{a,b\}:a\in A,b\in B,a<b\}=2$ whenever $A,B\in [\omega_1]^\omg$. Such functions were defined in \cite{Lspace}  from a ladder system on $\omega_1$ in a \emph{robust} way: if our model of set theory $V$ is extended to another model $W$ preserving $\omega_1$ then $f(a,b)$ evaluated in $V$ and $W$ agree. Now let $\overrightarrow{ab}$ if and only if $a<_{\mb S}b$ and $f(a,b)=0$.

Given $T\in [\mb S]^\omg$ and $n\in \oo$ at least 3, we need to find a copy of $\arr{C_n}$ in $G_{\mb S}[T]$.  $T$ as a subtree of $\mb S$ is still Suslin, and hence forcing with $T$ over our model $V$ preserves cardinals (by ccc) and introduces an uncountable set $A\subseteq T$ so that $H=G_{\mb S}[A]$ is complete.

Now, working in the larger model $V^T$, the function $f$ still witnesses $\omg\nrightarrow[\omg;\omg]^2_2$ and so
by Theorem \ref{Komgprop} we can find a directed $n$-cycle $v_0...v_{n-1}$ in $A$. However, the fact that  $v_0...v_{n-1}$ forms an $n$-cycle in $T$ is absolute (see the remark on $f$ earlier) so this must be true in our original model $V$ as well.

\end{proof}

Let us state (without presenting the proof) that another class of graphs defined from well behaved non-special trees admit similar orientations: suppose that $S\subseteq \omg$ is stationary and let ${\sigma(S)}$ denote the poset on $\{t\subseteq S:t \text{ is closed}\}$ where $s\leq t$ if and only if $s$ is an initial segment of $t$.

\begin{theorem}
$\displaystyle{\ENL{G_{\sigma(S)}}{\bigwedge_{3\leq n\in \oo}\arr{C_n}}{\oo}}$ for any stationary $S\subseteq \omg$. 
\end{theorem}

Indeed, one can combine the machinery of \cite{soukuptrees} and the $\diamondsuit^+$-argument from Theorem \ref{posrel} to prove this result. Note that neither $G_{\sigma(S)}$ nor $G_{\mb S}$ for $\mb S$ Suslin contains an uncountable complete subgraph.\\

Next, we prove that shift graphs defined on large enough sets have large dichromatic number. Let $\Sh_n(\lambda)$ denote the graph on vertices $[\lambda]^n$ and edges $\{\xi_i:i<n\}\{\xi_j:1\leq j<n+1\}$ where $\xi_0<\dots<\xi_n\in \lambda$.

\begin{theorem}\label{Shiftthm}
$\ENL{\Sh_n(\exp_n(\kappa))}{\arr{C}_4}{\kappa}$ for all $2\leq n<\oo$. In particular, $$\dchr{\Sh_n(\exp_n(\kappa))}>\kappa.$$
\end{theorem}

As $\Sh_n(\lambda)$ has no odd cycles of length less than $2n$, we get:

\begin{cor}There are digraphs with arbitrary large dichromatic number and large odd (undirected) girth.
\end{cor}

\begin{cor} Any digraph $F$ which embeds into all digraphs $D$ with $\chr D>\oo$ must be bipartite.
\end{cor}

Lastly, we encourage the reader to keep the $n=2$ and $\kappa=\oo$ case in mind when reading the following proof; otherwise the technical details might overshadow the actual ideas involved.

\begin{proof}[Proof of Theorem \ref{Shiftthm}]
Let $\lambda= \exp_n(\kappa)$. We construct an orientation $D$ of $Sh_n(\lambda)$ so that whenever $G:[\lambda]^n\to \kappa$ then there is a monochromatic directed $4$-cycle. In particular, we aim for a copy of $\arr{C}_4$ of the following form: the vertices will be $\{\alpha_0\}\cup R,R\cup \{\beta\},\{\alpha_1\}\cup R,R\cup \{\alpha_3\}$ where $|R|=n-1$ and $\alpha_0<\alpha_1<R<\alpha_3<\beta$. 

 List all  pairs $(A,g)$ where $A\in[\lambda]^{\exp_{n-1}(\kappa)}$, $g:[A]^n\to \kappa$ as $\{(A_\beta,g_\beta):\beta\in S^\lambda_{\kappa^+}\}$ so that $\sup A_\beta<\beta$.

By induction on $\beta$ define the orientation of edges of the form $\{\alpha\}\cup R, R\cup \{\beta\}$ where $\alpha<R<\beta$ and $|R|=n-1$. In short, the $n+1$-tuple $\{\alpha\}\cup R\cup\{\beta\}$ will be oriented either up (meaning ${\{\alpha\}\cup R,R\cup \{\beta\}}\in E(D)$) or down (meaning ${R\cup \{\beta\},\{\alpha\}\cup R}\in E(D)$).

For notational simplicity we will use $\alpha R$, $R\beta$, $\alpha R\beta$ for $\{\alpha\}\cup R$, $R\cup \{\beta\}$ and $\{\alpha\}\cup R\cup\{\beta\}$ respectively.

Now fix $\beta\in S^\lambda_{\kappa^+}$ and $R\in [\beta]^{n-1}$. We define by induction on $i<\kappa$ disjoint finite sets $a_{\beta,R,i}\in A_\beta\cap \min(R)$ and direct the $n+1$-tuples of the form $\alpha R\beta$ where $\alpha\in a_{\beta,R,i}$.

Given $i<\kappa$ and the finite sets $a_{\beta,R,j}$ for $j<i$, we consider three cases:

\textbf{Case 1.} If there is $\alpha_0<\alpha_1\in  (A_\beta\cap \min(R))\setm \bigcup\{a_{\beta,R,j}:j<i\}$ and $\alpha_3\in A_\beta\setm \max(R)$ so that 
\begin{enumerate}
 \item $\alpha_1R \alpha_3$ and $\alpha_0 R \alpha_3$ are oriented differently (one up, other down), and
\item $g_\beta$ is constant $i$ on $\alpha_0 R,\alpha_1 R,R\alpha_3$.
\end{enumerate}
Then we let $a_{\beta,R,i}=\{\alpha_0,\alpha_1\}$ and define the orientation of $\alpha_0 R\beta$ and $\alpha_1 R\beta$ so that $\alpha_0 R,R\beta,\alpha_1 R,R\alpha_3$ is a copy of $\arr{C}_4$. 

\textbf{Case 2.} Suppose that Case 1 fails but there is   $\alpha_0<\alpha_1\in A_\beta\cap \min (R)\setm \bigcup\{a_{\beta,R,j}:j<i\}$ so that 
 $g_\beta$ is constant $i$ on $\alpha_0 R,\alpha_1 R$.
Then we let $a_{\beta,R,i}=\{\alpha_0,\alpha_1\}$ and define the orientation of $\alpha_1 R\beta$ and  $\alpha_0 R\beta$ differently. 

\textbf{Case 3.} If both Case 1 and Case 2 fails then we let $a_{\beta,R,i}=\emptyset$.

This finishes the induction on $i<\kappa$ and in turn completes the definition of $D$.

Now, suppose that $G:[\lambda]^n\to \kappa$  and our aim is to find a monochromatic $4$-cycle. Take a $\kappa$-closed elementary submodel $M$ of size $\exp_{n-1}(\kappa)$ so that $D,G\in M$ and $\exp_{n-1}(\kappa),\mc{P}^k(\kappa)\subseteq M$ for $k\leq n-1$. Here $\mc P$ is the power set operator, $\mc P^0(\kappa)=\kappa$ and $\mc P^{k+1}(\kappa)=\mc P(\mc P^k(\kappa))$.

 Find $\beta  \in S^\lambda_{\kappa^+}$ so that $$(A_\beta,g_\beta)=(M\cap \lambda,G\uhr [M\cap \lambda]^2).$$

Now we define a sequence of maps $G_0,G_1\dots G_{n-1}$ so that $$G_{n-k}:[\lambda]^{k}\to \mc P^{n-k}(\kappa)$$ as follows. We define $G_0:[\lambda]^n\to \kappa$ simply by $G_0 =G$. Next, we define $G_1:[\lambda]^{n-1}\to \mc P(\kappa)$ by 

\begin{equation}
G_1(x_1,\dots,x_{n-1})=\{i\in \kappa:|\{\xi<x_1:G_0(\xi,x_1,\dots x_{n-1})=i\}|\geq \exp_{n-1}(\kappa)\}\in \mc P(\kappa).
\end{equation}

In general, given $G_{n-k-1}$, we let 
\begin{equation}
G_{n-k}(x_{n-k},\dots,x_{n-1})=\{i\in \mc P^{n-k-1}(\kappa):|\{\xi<x_{n-k}:G_{n-k-1}(\xi,x_{n-k},\dots x_{n-1})=i\}|\geq \exp_{k}(\kappa)\}.
\end{equation}
Finally, for $k=1$, we let

\begin{equation}
G_{n-1}(x_{n-1})=\{i\in \mc P^{n-2}(\kappa):|\{\xi<x_{n-1}:G_{n-2}(\xi,x_{n-1})=i\}|\geq \exp_{1}(\kappa)\}\in \mc P^{n-1}(\kappa).
\end{equation}

Note that $G_{n-1}(\beta)\in M$ by the assumptions on $M$. 

\begin{claim}There is a decreasing sequence of ordinals $\xi_{n-1},\xi_{n-2}\dots \xi_0$ and $\in$-decreasing $i_{n-1},i_{n-2}\dots i_0$ with the following properties:
\begin{enumerate}
 \item $\xi_{n-1}=\beta$ and $i_{n-1}=G_{n-1}(\beta)$,
\item $\xi_{n-2}\in A_\beta \cap \xi_{n-1}=A_\beta$ so that 
\begin{enumerate}
\item $G_{n-1}(\xi_{n-2})=i_{n-1}=G_{n-1}(\xi_{n-1})$, and
\item $i_{n-2}=G_{n-2}(\xi_{n-2},\xi_{n-1})\in i_{n-1}$;
\end{enumerate}
\item in general, $\xi_{n-k-1}\in A_\beta\cap \xi_{n-k}$ so that 
\begin{enumerate}
\item $G_{n-k}(\xi_{n-k-1}\dots\xi_{n-2})=i_{n-k}=G_{n-k}(\xi_{n-k}\dots\xi_{n-1})$, and
\item $i_{n-k-1}=G_{n-k-1}(\xi_{n-k-1},\xi_{n-k}\dots\xi_{n-1})\in i_{n-k}$
\end{enumerate}
where $k=0\dots n-1$.
\end{enumerate}
\end{claim}
\begin{proof} Given $\xi_{n-1}=\beta$ and $i_{n-1}=G_{n-1}(\beta)$, observe that $G\in M,\ran(G_{n-1})\subseteq M$ implies that $$\Lambda=\{\xi\in S^\lambda_{\kappa^+}:G_{n-1}(\xi)=G_{n-1}(\beta)\}\in M$$ as well. Hence $|\Lambda|\geq (\exp_{n-1}(\kappa))^+$ (as $\beta\in \Lambda\setm M$) and so $|\Lambda\cap M|=\exp_{n-1}(\kappa)$. In turn, $cf(\exp_{n-1}(\kappa))>\exp_{n-2}(\kappa)$ implies that there is an $i_{n-2}\in \mc P^{n-2}(\kappa)$ so that
\begin{equation}\label{rholine}
|\{\xi\in \Lambda\cap M: G_{n-2}(\xi,\xi_{n-1})=i_{n-2}\}|\geq \exp_{n-2}(\kappa).
\end{equation}
In particular, $i_{n-2}\in i_1=G_{n-1}(\xi_{n-1})$. Now pick any $\xi_{n-2}\in \Lambda\cap M$ with $G_{n-2}(\xi_{n-2},\xi_{n-1})=i_{n-2}$. 

Suppose we found $\xi_{n-k}$ and $i_{n-k}$ as required. By assumption, $i_{n-k}\in i_{n-k+1}=G_{n-k+1}(\xi_{n-k}\dots \xi_{n-2})$. Hence $$|\{\xi<\xi_{n-k}:G_{n-k}(\xi,\xi_{n-k}\dots \xi_{n-2})=i_{n-k}\}|\geq \exp_{k}(\kappa)$$ and so 
$$|\{\xi\in M\cap \xi_{n-k}:G_{n-k}(\xi,\xi_{n-k}\dots \xi_{n-2})=i_{n-k}\}|\geq \exp_{n-k}(\kappa)$$  holds as well. By cofinality considerations, there is $i_{n-k-1}\in \mc P^{n-k-1}(\kappa)$ so that $$|\{\xi\in M\cap \xi_{n-k}:G_{n-k}(\xi,\xi_{n-k}\dots \xi_{n-2})=i_{n-k},G_{n-k-1}(\xi,\xi_{n-k}\dots \xi_{n-1})=i_{n-k-1}\}|\geq \exp_{n-k}(\kappa).$$ 

Note that this implies that $i_{n-k-1}\in i_{n-k}=G_{n-k}(\xi_{n-k}\dots \xi_{n-1})$ and we can pick any $\xi_{n-k-1} \in M\cap \xi_{n-k}$ so that $G_{n-k}(\xi_{n-k-1},\xi_{n-k}\dots \xi_{n-2})=i_{n-k},G_{n-k-1}(\xi_{n-k-1},\xi_{n-k}\dots \xi_{n-1})=i_{n-k-1}$. Hence conditions (a) and (b) above are satisfied.
\end{proof}

At last we get $\xi_{0}\in A_\beta\cap \xi_{1}$ and $i_0\in \kappa$ so that 
\begin{enumerate}[(a)]
\item $G_{1}(\xi_{0}\dots\xi_{n-2})=i_{1}=G_{1}(\xi_{1}\dots\xi_{n-1})$, and
\item $i_{0}=G_{0}(\xi_{0}\dots\xi_{n-1})\in i_{1}$.
\end{enumerate}

We let $R=\{\xi_0\dots\xi_{n-2}\}\in [A_\beta]^{n-1}$ and look at the construction of the orientation $D$ when we considered $\beta$ and $R$. In particular, we consider the step when the colour $i_0\in \kappa$ came up.

If the assumption of Case 1 was satisfied then we constructed a copy of $\arr{C}_4$ on vertices $\alpha_0 R,R\beta,\alpha_1 R,R\alpha_3$ and $g_\beta$ is constant $i_0$ on $\alpha_0 R,\alpha_1 R,R\alpha_3$; $g_\beta$ and $G$ agree on $A_\beta$ and $G(R\beta)=i_0$ so this is the desired monochromatic  $\arr{C}_4$.

Now, we suppose that Case 1 fails and reach a contradiction. We claim that the assumption of Case 2 is satisfied. Indeed, $i_0\in i_1=G_{1}(\xi_{0}\dots\xi_{n-2})$ implies that there are $\kappa$ many $\xi<\xi_0$ with the property that $G(\xi,\xi_0\dots\xi_{n-2})=i_0$. Also, $M$ is $\kappa$-closed so $\bigcup\{a_{\beta,R,j}:j<i\}\in M$. Hence, using elementarity, we can select $\alpha_0<\alpha_1\in M\cap \xi_0\setm \bigcup\{a_{\beta,R,j}:j<i_0\}$ so that $$G(\alpha_0,\xi_0\dots\xi_{n-2})=G(\alpha_1,\xi_0\dots\xi_{n-2})=i_0.$$

Hence, following the instructions in Case 2, we oriented $\alpha_0 R\beta$ and $\alpha_1 R\beta$ differently. In turn, the set $$B=\{\xi\in \lambda\setm \max(R): \text{ the orientation of }\alpha_0 R\xi,  \alpha_1 R\xi \text{ are different and }G(R \xi)=i_0\}$$

is not {\DD empty}. However, $B\in M$ so we can choose $\alpha_3\in B\cap M$ and hence $\alpha_1,\alpha_2,\alpha_3$ witnesses that Case 1 holds. This contradicts our assumption and ends the proof.

\end{proof}

Now, we prove that consistently \textit{any graph} with size and chromatic number $\omg$ has uncountable dichromatic number as well in a very strong sense. 

Recall that $\diamondsuit^+$ asserts the existence of sets $\mc S_\beta=\{S_{\beta,n}:n\in\omega\}$ where $\beta\in \omg$ so that for every $X\subseteq \omg$ there is a club $C\subseteq \omg$ such that $X\cap \beta= S_{\beta,n}$ for some $n\in\oo$ whenever $\beta\in C$. 

\begin{theorem}\label{posrel} Assume that $\diamondsuit^+$ holds and the graph $G$ has size and chromatic number $\omg$. Then  $$\ENLL{G}{\bigwedge \{D:D\text { is an orientation of }\half\}}.$$ 
\end{theorem}

In other words, there is an orientation $D^*$ of $G$ so that whenever $G[A]$ is uncountably chromatic and $D$ is an orientation of $\half$ then $D$ embeds into $D^*[A]$.

\begin{proof} Let $\mc S_\beta=\{S_{\beta,n}:n\in\omega\}$ be the $\diamondsuit^+$ sequence and let $\{D_\beta:\beta<\omg\}$ list all orientations of $H_{\omega,\omega}$. We suppose that each $D_\beta$ has vertices $2\times \omega$. 

 By induction on $\beta$ we orient the edges of $G$ from $\beta$ to $\beta\cap N(\beta)$ to define an orientation $D^*_{\beta+1}$ of $G[\beta+1]$. 

Let $$\{(\delta_j,n_j,\xi_j):j<\omega\}$$ list  $\Gamma_\beta=\{(\delta,n,\xi):\delta,\xi\leq\beta,n<\omega,|N(\beta)\cap S_{\delta,n}|=\omega\}$. 

 Suppose that the orientation $D^*_\beta$ of $G[\beta]$ is defined already. By induction on $j<\oo$, we will select $a_j\in [N(\beta)\cap S_{\delta_j,n_j}]^{<\omega}$ such that $a_j\cap a_{j'}=\emptyset$ for $j<j'<\omega$ and orient the edges between $\beta$ and $a_j$ as follows.


At step $j$, we look at the set $\Phi_j$ of partial digraph embeddings $\varphi$ of $D_{\xi_j}$ into $D^*_\beta[S_{\delta_j,n_j}]$ such that $\dom(\varphi)=2\times k$ for some $k\leq\omega$ and  $$\varphi[\{0\}\times k]\subseteq  S_{\delta_j,n_j}\cap N(\beta)\setm \bigcup_{j'<j}a_j.$$ Note that $\Phi_j$ might only contain $\emptyset$. In any case, take a $\varphi_j\in \Phi_j$ which is maximal with respect to inclusion.

\textbf{Case 1}: If $\varphi_j$ is a complete embedding  then let $a_j=\emptyset$ and move to step $j+1$ in the induction. 


\textbf{Case 2:} if Case 1 fails then there is a $k<\omega$ such that $\dom(\varphi_j)=2\times k$ i.e. $\varphi_j$ is a digraph embedding of $D_{\xi_j}[2\times k]$. Let $a_j=\varphi_j[\{0\}\times k]\cup \{\alpha\}$ for some $$\alpha\in S_{\delta_j,n_j}\cap N(\beta)\setm (\bigcup_{j'<j}a_j\cup \varphi_j[\{0\}\times k]).$$ Lets define the orientation between $\beta$ and $a_j$ so that $$\varphi^*_j=\varphi_j\cup \{((0,k),\alpha),((1,k),\beta)\}$$ is a digraph embedding of $D_{\xi_j}[2\times (k+1)]$.

Edges from $\beta$ to $(N(\beta)\cap\beta)\setm \bigcup\{a_j:j<\omega\}$ are oriented arbitrarily. This finishes the definition of the orientation $D^*=\bigcup_{\beta<\omg} D^*_\beta$ of $G$.


Now take any $A$ such that $G[A]$ is uncountably chromatic and an orientation $D$ of $\half$. There is a club of elementary submodels $\{M_\alpha:\alpha<\omega_1\}$ of $H(\aleph_2)$ so that $D,A\in M_\alpha$ and whenever $\delta=M_\alpha\cap \omg$ for some $\alpha<\omg$ then $A\cap \delta=S_{\delta,n}$ for some $n\in \omega$.

 As $G[A]$ is uncountably chromatic, there is $\delta=M_\alpha\cap \omg$ for some $\alpha<\omg$ and $\beta\in A\setm \delta$ such that $N(\beta)\cap A\cap \delta$ is infinite. We can suppose that there is a $\xi\leq \beta$ so that $D=D_\xi$. Let $n<\omega$ such that $A\cap \delta=S_{\delta,n}$; so $N(\beta)\cap S_{\delta,n}$ is infinite. We claim that there is a copy of $D$ in $D^*[S_{\delta,n}]\subseteq D^*[A]$.


Let us look at how the orientation was defined between $\beta$ and $N(\beta)\cap \beta$. There is a $j<\omega$ such that $(\delta_j,n_j,\xi_j)=(\delta,n,\xi)$. 

At step $j$, we selected a maximal partial embedding $\varphi_j$ of $D$ into $D^*_\beta[S_{\delta,n}]$. If Case 1 held for $\varphi_j$ then $\varphi_j$ is a complete embedding witnessing the existence of a copy of $D$ in $D^*[S_{\delta,n}]$.

 Otherwise, we are in Case 2: $\varphi_j$ is an embedding of $D[2\times k]$ for some $k<\omega$. This case, we extended $\varphi_j$ into a strictly larger embedding $\varphi^*_j$ with $$\ran(\varphi^*_j)\subseteq S_{\delta_j,n_j}\cup \{\beta\}\subseteq A.$$ Note that $\varphi_j$ and $A$ are both in $M_\alpha$ and $$H(\aleph_2)\models  \varphi_j \text{ can be extended to an embedding of }D[2\times (k+1)] \text{ into } A\setm \bigcup_{j'<j}a_j.$$ Hence this sentence must be true in $M_\alpha$ as well, that is, there is an embedding $\varphi\in M_\alpha$ of $D[2\times (k+1)]$ into $A\setm \bigcup_{j'<j}a_j$  extending $\varphi_j$. Of course the range of $\varphi$ now has to be in $A\cap M_\alpha=S_{\delta,n}$ which contradicts the maximality of $\varphi_j$.

\end{proof}

\begin{cor} If $\diamondsuit^+$ holds and $G$ has size $\omg$ then  $\dchr{G}=\omega_1$ if and only if $\chr G=\omg$.
\end{cor}

\section{On the lack of orientations with large chromatic number} \label{negrelsec}

{\DD In this final section, we show that consistently there is a graph $G$ with uncountable chromatic number without an orientation with uncountable dichromatic number. In other words, $\chr G>\oo$ does not imply $\dchr D>\oo$.}

In \cite{simchrom}, Hajnal and Komj\'ath study an intriguing problem which can be roughly stated as follows: given a graph $G$ with uncountable chromatic number, can we colour the edges of $G$ with 2 (alternatively, $\oo$ or $\omg$) colours so that each colour appears on each \emph{large enough} subgraph.

Let us observe some straightforward connections to our investigations: 

 \begin{obs} \label{simchromobs}\begin{enumerate}
	 \item If $\ENLL{G}{\bigvee\{\arr{C}_n:3\leq n<\oo\}}$ then we can define a 2-edge colouring of $G$ with the property that every colour appears on every uncountably chromatic induced subgraph.
\item If  $\dchr{G}>\oo$ then there is a 2-colouring of the edges of $G$ such that whenever the vertices $V(G)$ are partitioned into countably many pieces $\{V_i:i<\omega\}$ then both colours appear on one of the spanned subgraphs $G[V_i]$ 
 \end{enumerate} 
\end{obs}


\begin{proof} Let us prove (1) and leave the completely analogous proof of (2) to the reader. Given the orientation $D$ of $G=(\lambda,E)$ witnessing $\ENLL{G}{\bigvee\{\arr{C}_n:3\leq n<\oo\}}$ define $f(ab)=0$ if $a<b\in \lambda$ and $\arr{ab}\in E(D)$, otherwise $f(ab)=1$. Now, if $\chr{G[W]}>\oo$ then there is a directed cycle $C$ in $D[W]$.  Let $b=\max C$ (where $C$ is considered as a set of ordinals in $\lambda$). If $a$ and $a'$ are the neighbours of $b$ in $C$ then we must have $f(ab)\neq f(a'b)$. 
\end{proof}


It was shown in \cite[Theorem 5]{simchrom} that the consequence stated in Observation \ref{simchromobs} (1) can consistently fail for a graph of size and chromatic number $\omega_1$. Hence, we get the following:

\begin{cor} Consistently, there is a graph $G$ with size and chromatic number $\omega_1$ such that $\ENLL{G}{\bigvee\{\arr{C}_n:3\leq n<\oo\}}$ fails.
\end{cor}

However, it is unknown if there is a graph $G$ of chromatic number $\omg$ which fails the consequence stated in Observation \ref{simchromobs} (2) (even consistently). Hence it is rather interesting that $\dchr{G}\leq \oo$ is possible for a graph $G$ with uncountable chromatic number:
{\DD
\begin{theorem}\label{negrel} Consistently, there is a graph $G$ on vertex set $\omg$ such that
\begin{enumerate}
 \item $\dchr{G}\leq \omega$, however
\item $C_3\hookrightarrow G[X]$ for every uncountable $X\subseteq \omg$, and so $\chr G=\omg$.
\end{enumerate}


\end{theorem}}

{\DD At this point, we don't know if the implication in  Observation \ref{simchromobs} (2) can or cannot be reversed.}

\begin{proof} We start from a model $V$ of CH. 

Let $\mb P_0=\{(s,g):s\in [\omg]^{<\omega},g\subseteq [s]^2\}$ with the usual ordering.

	Given a generic filter $\mc G\subseteq \mb P_0$ we get a graph $\dot G=\{g^p:p\in \mc G\}$ in the extension $V^{\mb P_0}$. 

We will not use this particular fact, but let us mention that $\dot G$ has uncountable dichromatic number in $V^{\mb P_0}$:

\begin{claim}
 $V^{\mb P_0}\models \ENLL{\dot G}{\arr{C_3}}$.
\end{claim}

\begin{proof}
 Let $f_\beta:\omega\to \beta$ be a bijection for $\beta\in \omg$. We define an orientation $\dot D$ as follows: if $\alpha<\beta<\omg$ and $\alpha\beta\in \dot G$ then ${\alpha\beta}\in E(\dot D)$ if and only if $\beta$ and $\beta+n$ are connected in $\dot G$ for  $\alpha=f_\beta(n)$; otherwise, $\beta\alpha\in E(\dot D)$.

 Suppose that $p_0\force \dot W$ is uncountable and let $p_0,\dot W,\mb P_0,(f_\alpha)_{\alpha<\omg} \in M_0\prec M_1$ where $M_0,M_1$ are countable elementary submodels of $H(\aleph_2)$. Find $\beta\in \omg\setm M_1$ and $p\leq p_0$ so that $p\force \beta \in \dot W$. Let $I=\{n:\beta+n\in s^p\}$ and find a $p'\in M_1$ so that $p$ and $p'$ are compatible and there is $\beta'\in s^{p'}\setm \ran(f_\beta\uhr I)$ with $p'\force \beta'\in \dot W$. 

Let $J=\{n:\beta'+n\in s^{p'}\}$ and find $p''\in M_0$ compatible with both $p$ and $p'$ such that there is $$\beta''\in s^{p''}\setm (\ran(f_\beta\uhr I)\cup\ran(f_{\beta'}\uhr J))$$ with $p''\force \beta''\in \dot W$. 

It is easy to see that we can find a common extension $q$ of $p,p'$ and $p''$ which forces that $\{\beta,\beta',\beta''\}$ is a copy of $\arr{C_3}$.
\end{proof}
Let us remark that the same argument shows  $V^{\mb P_0}\models \ENLL{\dot G}{D_0}$ for any finite digraph $D_0$.

	Now define a finite support iteration $(\mb P_\alpha,\dot{\mb Q}_\beta)_{\alpha\leq \omega_2,\beta<\omega_2}$ as follows: in $V^{\mb P_\alpha}$, we consider an orientation $\dot D_{\alpha}$ of $\dot G$ and let $$\dot{\mb Q}_\alpha=\{q\in Fn(\omega_1,\omega): \dot D_\alpha[q^{-1}(n)] \text{ is acyclic for all }n\in \omega\}.$$ 
	
	Recall that $Fn(\lambda,\kappa)$ denotes the set of all finite partial functions from $\lambda$ to $\kappa$.
	
	Clearly, the forcing $\dot{\mb Q}_\alpha$ introduces an $\omega$-partition $\{\dot W_n:n<\omega\}$ of $\omega_1$ such that  $\dot D_\alpha[\dot W_n]$ has no directed cycles. In turn, $V^{\mb P_\alpha*\dot Q_\alpha}\models \chr{\dot D_\alpha}\leq \oo.$
	
	Our goal is to prove that this iteration is ccc and in the final model $V^{\mb P_{\omega_2}}$ our graph $\dot G$ is uncountably chromatic.

	\begin{tclaim} The set of $p\in \mb P_{\alpha}$ which satisfy
	\begin{enumerate}
		\item $p(\xi)\in V$ and $\dom(p(\xi))\subseteq s^{p(0)}$, and
		\item $p\uhr \xi$ decides the orientation of $\dot D_\xi$ on $\dom(p(\xi))$ 
	\end{enumerate}	for all $\xi\in \supp(p)$ is dense in $\mb P_{\alpha}$.
	\end{tclaim}
	We will call these conditions \emph{determined} and we only work with determined conditions from now on if not mentioned otherwise.
	\begin{proof} Easy induction on $\alpha$.
	\end{proof}
	
	We say that $p,p'\in \mb P_0$ are \emph{twins} if $|s^p|=|s^{p'}|$, $s^{p}\cap s^{p'}<s^{p}\setm  s^{p'}< s^{p'}\setm  s^{p}$ (or vica versa $s^{p'}\setm  s^{p}< s^{p}\setm  s^{p'}$) and the unique monotone bijection $\psi_{p,p'}:s^{p}\to s^{p'}$ (which fixes $s^{p}\cap s^{p'}$) gives a graph isomorphism between the graphs $p$ and $p'$. We say two conditions from the iteration $q,q'\in \mb P_\alpha$ are \emph{twins} if \begin{enumerate}
		\item $p(0)$ and $p'(0)$ are twins,
\end{enumerate}  and for all $\xi\in \supp(p)\cap \supp(p')$
\begin{enumerate}
\setcounter{enumi}{1}		
\item $\dom(p'(\xi))=\psi_{p,p'}[\dom(p(\xi))]$,
		\item  $p(\xi)(\delta)=p'(\xi)(\psi_{p,p'}(\delta))$ for all $\delta\in \dom(p(\xi))$, and
		\item\label{isocond} $p\uhr \xi\force {a_0 a_1}\in \dot D_\xi$ if and only if $p'\uhr \xi\force {a'_0 a_1'}\in \dot D_\xi$ for all $\xi\in \supp(p)\cap \supp(p')$ where $a_i'=\psi_{p,p'}(a_i)$.
	\end{enumerate} 
	
		\begin{tclaim}\label{twinclaim} If $p,p'\in \mb P_\alpha$ are determined and  twins then they have a minimal common extension $p\vee p'$.
	\end{tclaim}
	\begin{proof} We define $p\vee p'$ by 
		 \[
    (p\vee p')(\alpha) = \begin{cases}
		     (s^p\cup s^{p'},g^p\cup g^{p'}), & \text{for } \alpha=0,\text{ and}\\
        p(\xi)\cup p'(\xi), & \text{for } \xi\in\alpha \setm\{0\}.
   
        \end{cases}
  \]
	It is clear that $(p\vee p')(0)\in \mb P_0$ so let us show $(p\vee p')\uhr \xi$ forces that there are no monochromatic cycles with respect to $p(\xi)\cup p'(\xi)$. We do this by induction on $\xi<\alpha$. 

Suppose that $p\vee p'\uhr \xi$ is a condition. If $\xi\in \supp(p)\setm \supp(p')$ or $\xi\in \supp(p')\setm \supp(p)$ then $p\vee p'\uhr \xi+1\in \mb P_{\xi+1}$  as well.
	

Now let $\xi\in \supp(p)\cap \supp(p')$. We need to show that $p\vee p'\uhr \xi$ forces that there are no monochromatic directed cycles in $\dot D_\xi$ with respect to $(p\vee p')(\xi)$. This will easily follow from Lemma \ref{amalglemma}: fix $k<\oo$ and consider $D=\dot D_\xi[\{v\in s^p:p(\xi)(v)=k\}$ and $D'=\dot D_\xi[\{v\in s^{p'}:p'(\xi)(v)=k\}$. $p\vee p'\uhr \xi$ forces that $D$ and $D'$ are isomorphic, acyclic digraphs satisfying the assumptions of Lemma \ref{amalglemma} as $p$ and $p'$ are twins. Hence, by Lemma \ref{amalglemma} (2), $D\cup D'=\dot D_\xi[\{v\in s^{p\vee p'}:(p\vee p')(\xi)(v)=k\}$ is acyclic as well.

\end{proof}

	\begin{tclaim} $\mb P_{\alpha}$ is ccc for all $\alpha\leq \omega_2$.
	\end{tclaim}
	\begin{proof}Indeed, the $\Delta$-system lemma implies that any uncountable set of conditions must contain an uncountable set of pairwise twins, and Claim \ref{twinclaim} implies that any two twin conditions are comparable.
	\end{proof}
	
	Now, by a standard bookkeeping argument, we can choose the names $\dot D_{\alpha}$ of $\dot G$ so that $$V^{\mb P_{\omega_2}}\models \dchr{G}\leq \oo.$$ 
	
Before we show that the chromatic number of $G$ is still uncountable after the iteration, let us emphasize a simple fact which again follows Lemma \ref{amalglemma} (1):

\begin{tclaim}\label{nodirpath} Suppose that $p,p'\in \mb P_{\omega_2}$ are determined and  twins and $\delta\in s^p\setm s^{p'}$. Then $p\vee p'\uhr \xi$ forces that there is no directed path from $\delta$ to $\delta'=\psi_{p,p'}(\delta)$ in $\dot D_\xi$ which is monochromatic with respect to $p\vee p'(\xi)$.
	\end{tclaim}
	
Now, we can prove the following:

\begin{tclaim}\label{amalg} Suppose that $p,p'\in \mb P_{\omega_2}$ are determined and twins and $\delta\in s^p\setm s^{p'}$. Then there is a minimal extension $p\vee_\delta p'$ of $p\vee p'$ which forces that $\delta$ is connected to $\delta'=\psi_{p,p'}(\delta)$ in $\dot G$.
	\end{tclaim}
\begin{proof}
 We define $q=p\vee_\delta p'$ by
	 \[
    q(\alpha) = \begin{cases}
		     (s^{p\vee p'},g^{p\vee p'}\cup \{\delta,\delta'\}), & \text{for } \alpha=0, \text{ and}\\
        (p\vee p')(\alpha), & \text{for } \alpha\in\omega_2\setm\{0\}.
   
        \end{cases}
  \]
That is, we essentially take $p\vee p'$ and add the single edge $\{\delta,\delta'\}$. If we prove that $q$ is a condition then we are done. 

We prove that $q\uhr \xi$ is a condition by induction on $\xi<\omega_2$. Suppose we proved that $q\uhr \xi$ is a condition but there is some $r\leq q\uhr \xi$ and $C$ so that $r$ forces that $C$ is a monochromatic cycle with respect to $q(\xi)$ in $\dot D_\xi$. $C$ must contain the edge $\{\delta,\delta'\}$ and so $r$ forces that $\delta$ and $\delta'$ are connected by a directed monochromatic path $P$ (not containing the edge $\{\delta,\delta'\}$). However, orientation and colouring on $P$ is decided by $p\vee p'\uhr \xi$ already, which contradicts {\DD Claim} \ref{nodirpath}.

\end{proof} Note that $p\vee_\delta p'$ is not  {\DD necessarily} determined.	In any case, to finish our proof, it suffices to show the following:

{\DD
	
		\begin{tclaim} $V^{\mb P_{\omega_2}}\models C_3\hookrightarrow \dot G[\dot X]$ for every uncountable $\dot X\subseteq \omg$.
	\end{tclaim}
	\begin{proof} Suppose $p\force \dot X$ is uncountable for some $p\in{\mb P_{\omega_2}}$; we will find $q\leq p$ and $\beta,\beta',\beta''\in \omg$ such that $q\force \{\beta,\beta',\beta''\}$ is a triangle in $\dot G[\dot X]$.

As ${\mb P_{\omega_2}}$ is ccc, there is an uncountable set $Y\subseteq \omg$ in $V$ and determined conditions $p_\beta\leq p$ so that $\beta\in s^{p_\beta}$ and $p_\beta\force \beta\in \dot X$ whenever $\beta\in Y$.  


Take a countable elementary submodel $M$ of some large enough $H(\theta)$ containing everything relevant and let $\beta\in Y\setm M$. We let $s=s^{p_\beta}\cap M$. Using elementarity, find $\beta'\in Y\cap M$ so that 
\begin{enumerate}
                                                                                                                                          	                                                                                                     \item $p_\beta$ and $p_{\beta'}$ are twins,
\item $s=s^{p_\beta}\cap s^{p_{\beta'}}$, and
\item $\beta=\psi_{p',p}(\beta')$.                                                                                          \end{enumerate} Let $r$  be a determined condition extending $p_{\beta'}\vee_{\beta'} p_{\beta}$. 

Now, let $\tilde s=s^{r}\cap M$ and, using elementarity again, find a condition $r'\in M$ so that 
\begin{enumerate}
                     \setcounter{enumi}{3}	                      	                                   \item $r$ and $r'$ are twins,
\item $\tilde s=s^{r}\cap s^{r'}$, and
\item $\beta=\psi_{r',r}(\beta'')$.                                                                                          \end{enumerate} Finally, let $q=r'\vee_{\beta''} r$. Clearly, $q\force ``\{\beta,\beta',\beta''\}$ is a triangle in $\dot G[\dot X]$".
\end{proof}}
\end{proof}

{\DD In particular, if one is looking for the value of $f(\aleph_1)$ from Conjecture \ref{ENLconj}, it needs to be larger than $\aleph_1$, at least consistently.}

\section{Open problems}\label{probsec}

It seems that it is non trivial to find in ZFC a single digraph $D$ with uncountable dichromatic number. Indeed, even to show that $K_\omg$ has an orientation with uncountable dichromatic number required the application of $\omg\nrightarrow[\omg;\omg]^2_2$ which is a deep result of J. Moore \cite{Lspace}. Hence, we have the following meta-problem:

\begin{prob} Provide a simple/elementary proof of the fact that  $K_\omg$ has an orientation with uncountable dichromatic number.
\end{prob}

An obvious question which comes to mind regarding the definition of $\dchr G$ is whether the $\sup$ is actually a $\max$:

\begin{quest} Suppose that an undirected graph $G$ has orientations $D_\xi$ for $\xi<\cf(\kappa)$ so that $\sup\{\chr{D_\xi}:\xi<\cf(\kappa)\}=\kappa$. Is there a single orientation $D$ of $G$ so that $\chr D=\kappa$?
\end{quest}

We conjecture that the answer is yes if $\cf(\kappa)=\oo$ but no if $\cf(\kappa)>\oo$.
\medskip

Regarding obligatory subgraphs and Proposition \ref{halfembed}, we ask:

\begin{prob} Characterize those orientations $D^*$ of $\half$ so that $D^*$ embeds into any digraph $D$ with uncountable dichromatic number.
\end{prob} 

At this point, we don't even have a list of those digraphs say on 4 vertices which embed into any digraph $D$ with uncountable dichromatic number.

\subsection{More on cycles} It is known that $\chr G>\omega$ implies that $G$ has cycles of all but finitely many lengths \cite{EH0, thomassen}. 

\begin{quest} Are there digraphs $D$ with $\chr D>\omega$ so that $\arr C_k$ does not embed into $D$ for infinitely many $k$?
\end{quest}

More generally, describe what sets of cycles can be omitted by $D$ with $\chr D>\omega$. In particular, answer the following:

\begin{quest} Does $D\to (\arr{C}_3)^1_\oo$ imply that $\arr C_4\hookrightarrow D$?
\end{quest}

The answer is yes for the undirected version: $G\to ({C}_3)^1_\oo$ obviously implies $\chr G>\oo$ and so $ C_4\hookrightarrow G$ and even $G\to ({C}_{2k})^1_\oo$ holds for all $2\leq k<\oo$.

It was shown by Erd\H os and R. Rado \cite{erdosrado} that triangle-free graphs $G$ with size and chromatic number $\kappa$ can be constructed without additional set theoretic assumptions.

\begin{quest} Is there in ZFC a digraph $D$ of size and dichromatic number $\omg$ such that $D$ has no directed triangles?
\end{quest}

 If we omit the requirement on size then the shift graphs provide an example by Theorem \ref{Shiftthm}.

Finally, it would be very interesting to construct digraphs as in Theorem \ref{digirth} in ZFC:

\begin{quest} 
Is there in ZFC, for every $k<\oo$, a digraph $D$ of uncountable dichromatic number such that $D$ has digirth at least $k$?
\end{quest}

\subsection{Connected subgraph}

It is clear that every graph with uncountable chromatic number has a connected component with uncountable chromatic number. Similarly:

\begin{obs}
 Suppose that $\chr D >\omega$. Then there is $D_0\subseteq D$ so that $\chr{D_0}>\omega$ and $D_0$ is strongly connected.
\end{obs}

Komj\'ath  \cite{kopeconn} showed that every graph $G$ with uncountable chromatic number contains a  $k$-connected subgraph with uncountable chromatic number where $k\in \oo$.

\begin{quest}\label{connquest} Suppose that $\chr D >\omega$ and $k\in \oo$. Is there a $D_0\subseteq D$ so that  $D_0$ is strongly $k$-connected?
\end{quest}

Even the case $k=2$ is open. In the undirected case, the balanced complete bipartite graph on $2k$ vertices is a $k$-connected subgraph of any graph $G$ with uncountable chromatic number. However, any strongly connected digraph $D_0$ contains directed cycles and hence, by Theorem \ref{digirth}, there is a digraph $D$ with uncountable dichromatic number so that $D_0$ does not embed into $D$. Hence no single strongly connected graph will be a universal witness providing a positive answer to Question \ref{connquest}.

If the answer is yes to Question \ref{connquest}, a more ambitious goal would be to find a $D_0\subseteq D$ so that $\chr{D_0}>\omega$ while $D_0$ is strongly $k$-connected.


 Regarding Theorem \ref{noarrowcon}, the most burning question is the following:

\begin{quest} Is there (even consistently) a digraph $D$ with $\chr{D}>\oo$ so that $\ENL{D}{\arr{C}_n}{2}$ fails for every $n\in \oo$?
\end{quest}

Naturally, any ZFC example would be very warmly welcome:

\begin{quest} Is there in ZFC a digraph $D$ with $\chr{D}>\oo$ so that $\ENL{D}{\arr{C}_n}{\oo}$ fails for every $n\in \oo$?
\end{quest}

\subsection{Various questions} Regarding the Erd\H os-Neumann-Lara problem, we ask:

\begin{prob}
Does $\chr G>2^\omega$ imply $\dchr G> \omega$?
\end{prob}

It would be rather natural to look into the following with regards to  Theorem \ref{posrel}:

\begin{prob}
Does $\chr G=\omg$ imply $\dchr G=\omg$ consistently (without restricting the size of $G$)?
\end{prob}

The following might be easier to answer:

\begin{prob}

Does $$ \ENL{G}{\arr{P_\omega}}{\omega}$$ hold in ZFC for $G$ with chromatic number $\omega_1$ where $\arr{P_\omega}$ is the one-way infinite directed path. 
\end{prob}

{\DD Finally, we close with a fascinating open problem of Neumann-Lara from 1985:

\begin{prob}
Does every planar digraph have dichromatic number at most 2?
\end{prob}

The answer is yes if the digirth is at most four \cite{harutplanar,li2016planar}. We mention that this is a problem on finite digraphs: if there is an infinite counterexample then it must have a finite subdigraph of dichromatic number greater than 2 as well by Theorem \ref{thm:compactness}.}


\begin{thebibliography}{10}

\bibitem{abrahamcard}
Uri Abraham and Menachem Magidor.
\newblock Cardinal arithmetic.
\newblock In {\em Handbook of set theory}, pages 1149--1227. Springer, 2010.

\bibitem{bokal}
Drago Bokal, Gasper Fijavz, Martin Juvan, P~Mark Kayll, and Bojan Mohar.
\newblock The circular chromatic number of a digraph.
\newblock {\em Journal of Graph Theory}, 46(3):227--240, 2004.

\bibitem{dutta}
Kunal Dutta and CR~Subramanian.
\newblock Improved bounds on induced acyclic subgraphs in random digraphs.
\newblock {\em SIAM Journal on Discrete Mathematics}, 30(3):1848--1865, 2016.

\bibitem{ENL}
Paul Erd{\H{o}}s.
\newblock Problems and results in number theory and graph theory.
\newblock In {\em Proc. 9th Manitoba Conf. Numer. Math. and Computing}, pages
  3--21, 1979.

\bibitem{erdos_dichrom}
Paul Erd{\H o}s, John Gimbel, and Dieter Kratsch.
\newblock Some extremal results in cochromatic and dichromatic theory.
\newblock {\em Journal of Graph Theory}, 15(6):579--585, 1991.

\bibitem{erdosrado}
Paul Erd{\H o}s and Richard Rado.
\newblock A construction of graphs without triangles having pre-assigned order
  and chromatic number.
\newblock {\em Journal of the London Mathematical Society}, 1(4):445--448,
  1960.

\bibitem{prob}
Paul Erd\H os.
\newblock Graph theory and probability.
\newblock {\em Canad. J. Math}, 11:34--38, 1959.

\bibitem{EH0}
Paul Erdős and Andr\'as Hajnal.
\newblock On chromatic number of graphs and set-systems.
\newblock {\em Acta Math. Acad. Sci. Hungar}, 17:61--99, 1966.

\bibitem{fodor}
G\'eza Fodor.
\newblock Proof of a conjecture of {P}. {E}rd{\H{o}}s.
\newblock {\em Acta Sci. Math. Szeged}, 14:219--227, 1951.

\bibitem{half}
Andr\'as Hajnal and P\'eter Komj\'ath.
\newblock What must and what need not be contained in a graph of uncountable
  chromatic number?
\newblock {\em Combinatorica}, 4(1):47--52, 1984.

\bibitem{simchrom}
Andr{\'a}s Hajnal and P{\'e}ter Komj{\'a}th.
\newblock Some remarks on the simultaneous chromatic number.
\newblock {\em Combinatorica}, 23(1):89--104, 2003.

\bibitem{harut}
Ararat Harutyunyan and Bojan Mohar.
\newblock Two results on the digraph chromatic number.
\newblock {\em Discrete Mathematics}, 312(10):1823--1826, 2012.

\bibitem{harutplanar}
Ararat Harutyunyan and Bojan Mohar.
\newblock Planar digraphs of digirth five are 2-colorable.
\newblock {\em Journal of Graph Theory}, 84(4):408--427, 2017.

\bibitem{kopeconn}
P\'eter Komj\'ath.
\newblock Connectivity and chromatic number of infinite graphs.
\newblock {\em Israel J. Math.}, 56(3):257--266, 1986.

\bibitem{kopesurv}
P\'eter Komj{\'a}th.
\newblock The chromatic number of infinite graphs--a survey.
\newblock {\em Discrete Mathematics}, 311(15):1448--1450, 2011.

\bibitem{kopeerdos}
P\'eter Komj{\'a}th.
\newblock Erdős’s work on infinite graphs.
\newblock {\em Erdős Centennial}, 25:325--345, 2014.

\bibitem{kunen}
K.~Kunen.
\newblock {\em Set theory an introduction to independence proofs}.
\newblock Elsevier, 2014.

\bibitem{li2016planar}
Zhentao Li and Bojan Mohar.
\newblock Planar digraphs of digirth four are 2-colourable.
\newblock {\em arXiv preprint arXiv:1606.06114}, 2016.

\bibitem{mohar2016}
Bojan Mohar and Hehui Wu.
\newblock Dichromatic number and fractional chromatic number.
\newblock In {\em Forum of Mathematics, Sigma}, volume~4. Cambridge University
  Press, 2016.

\bibitem{Lspace}
Justin Moore.
\newblock A solution to the {L}-space problem.
\newblock {\em Journal of the American Mathematical Society}, 19(3):717--736,
  2006.

\bibitem{nessparse}
Jaroslav Ne{\v{s}}et{\v{r}}il.
\newblock A combinatorial classic—sparse graphs with high chromatic number.
\newblock In {\em Erd{\H{o}}s Centennial}, pages 383--407. Springer, 2013.

\bibitem{nlara}
Victor Neumann-Lara.
\newblock The dichromatic number of a digraph.
\newblock {\em Journal of Combinatorial Theory, Series B}, 33(3):265--270,
  1982.

\bibitem{rinotrect}
Assaf Rinot and Stevo Todorcevic.
\newblock Rectangular square-bracket operation for successor of regular
  cardinals.
\newblock {\em Fund. Math}, 220(2):119--128, 2013.

\bibitem{severino}
Michael Severino.
\newblock A short construction of highly chromatic digraphs without short
  cycles.
\newblock {\em Contributions to Discrete Mathematics}, 9(2), 2014.

\bibitem{shelahcard}
Saharon Shelah.
\newblock {\em Cardinal arithmetic}.
\newblock Number~29. Oxford University Press on Demand, 1994.

\bibitem{soukuptrees}
D{\'a}niel~T. Soukup.
\newblock Trees, ladders and graphs.
\newblock {\em Journal of Combinatorial Theory, Series B}, 115:96--116, 2015.

\bibitem{soukel}
L.~Soukup.
\newblock Elementary submodels in infinite combinatorics.
\newblock {\em Discrete Math.}, 311(15):1585--1598, 2011.

\bibitem{spencer2008size}
Joel Spencer and CR~Subramanian.
\newblock On the size of induced acyclic subgraphs in random digraphs.
\newblock {\em Discrete Mathematics and Theoretical Computer Science}, 10(2),
  2008.

\bibitem{thomassen}
Carsten Thomassen.
\newblock Cycles in graphs of uncountable chromatic number.
\newblock {\em Combinatorica}, 3(1):133--134, 1983.

\end{thebibliography}
\end{document}